\documentclass[10.5pt,reqno]{article}
\usepackage{amsmath,amssymb,amsthm}
\usepackage{color}
\def\R{\mathbb R}  \def\C{\mathbb C} \def\N{\mathbb N}
\newtheorem{thm}{Theorem}[section]
\newtheorem{lem}[thm]{Lemma}
\newtheorem{cor}[thm]{Corollary}
\newtheorem{defn}[thm]{Definition}

\newtheorem{rem}[thm]{Remark}
\numberwithin{equation}{section}

\title{{\bf Continuous dependence of the Cauchy problem for the inhomogeneous nonlinear Schr\"{o}dinger equation in} $H^{s} (\R^{n} )$}
\author{{\bf JinMyong An, JinMyong Kim$^*$, KyuSong Chae}\\
\footnotesize{Faculty of Mathematics, {\bf Kim Il Sung} University, Pyongyang, Democratic People's Republic of Korea}\\
\footnotesize{$^*$ Corresponding Author}\\
\footnotesize{Email address: jm.kim0211@ryongnamsan.edu.kp}
}
\date{}
\begin{document}
\maketitle
\begin{abstract}
We consider the Cauchy problem for the inhomogeneous nonlinear Schr\"{o}dinger (INLS) equation
\[iu_{t} +\Delta u=|x|^{-b} f(u),\;u(0)\in H^{s} (\R^{n} ),\]
where $n\in \N$, $0<s<\min \{ n,\; 1+n/2\} $, $0<b<\min \{ 2,\;n-s,\;1+\frac{n-2s}{2} \} $ and $f(u)$ is a nonlinear function that behaves like $\lambda \left|u\right|^{\sigma } u$ with $\sigma>0$ and $\lambda \in \C$. Recently, An--Kim \cite{AK21} proved the local existence of solutions in $H^{s}(\R^{n} )$ with $0\le s<\min \{ n,\; 1+n/2\}$. However even though the solution is constructed by a fixed point technique, continuous dependence in the standard sense in $H^{s}(\R^{n} )$ with $0< s<\min \{ n,\; 1+n/2\}$ doesn't follow from the contraction mapping argument. In this paper, we show that the solution depends continuously on the initial data in the standard sense in $H^{s}(\R^{n} )$, i.e. in the sense that the local solution flow is continuous $H^{s}(\R^{n} )\to H^{s}(\R^{n} )$, if $\sigma $ satisfies certain assumptions.
\end{abstract}

\section{Introduction}

In this paper, we study the continuous dependence of the Cauchy problem for the inhomogeneous nonlinear Schr\"{o}dinger (INLS) equation
\begin{equation} \label{GrindEQ__1_1_}
\left\{\begin{array}{l} {iu_{t} +\Delta u=|x|^{-b} f(u)} \\ {u(0,\; x)=\phi (x),} \end{array}\right.
\end{equation}
where $\phi \in H^{s} \left(\R^{n} \right)$, $n\in \N$, $0<s<\min \left\{n,\;\frac{n}{2} +{1}\right\}$ and $0<b<2$ and $f$ is of class $C\left(\sigma ,s,b\right)$ (see Definition 1.1).

\begin{defn}\label{defn 1.1}
\textnormal{Let $f:\C\to \C$, $s>0$, $\sigma>0$, $0\le b<2$ and $\left\lceil s\right\rceil $ denote the minimal integer which is larger than or equals to $s$. For $k\in \N$, let $k$-th order derivative of $f(z)$ be defined under the identification $\C=\R^{2} $ (see Section 2). Let us define
\begin{equation} \label{GrindEQ__1_2_}
\sigma _{s} =\left\{\begin{array}{l} {\frac{4-2b}{n-2s} ,\;0<s<\frac{n}{2} ,} \\ {\infty ,\; s\ge \frac{n}{2} .} \end{array}\right.
\end{equation}
We say that $f$ is of class $C\left(\sigma ,s,b\right)$ if it satisfies one of the following conditions:}
\renewcommand{\theenumi}{\roman{enumi}}
\begin{enumerate}
\item  \textnormal{$f\left(z\right)$ is a polynomial in $z$ and $\bar{z}$ satisfying $1<\deg \left(f\right)=1+\sigma<1+\sigma_{s}$.}

\item  \textnormal{$f\in C^{\left\lceil s\right\rceil } \left(\C\to \C\right)$, $\left\lceil s\right\rceil -1<\sigma <\sigma_{s}$ and
\begin{equation} \label{GrindEQ__1_3_}
\left|f^{\left(k\right)} \left(u\right)\right|\lesssim\left|u\right|^{\sigma +1-k} ,
\end{equation}
for any $0\le k\le \left\lceil s\right\rceil $ and $u\in \C$. Furthermore,
\begin{equation} \label{GrindEQ__1_4_}
\left|f^{\left(\left\lceil s\right\rceil\right)} \left(u\right)-f^{\left(\left\lceil s\right\rceil\right)} \left(v\right)\right|\lesssim\left|u-v\right|^{\min \{ \sigma -\left\lceil s\right\rceil +1,\;1\} } \left(\left|u\right|+\left|v\right|\right)^{\max \{ 0,\;\sigma -\left\lceil s\right\rceil \} } ,
\end{equation}
for any $u,\;v\in \C$. If $s<1$, assume further that $f\in C^{2} \left(\C\to \C\right)$, $\sigma \ge 1$ and
\begin{equation} \label{GrindEQ__1_5_}
\left|f''\left(u\right)\right|\lesssim\left|u\right|^{\sigma -1},
\end{equation}
for any $u\in \C$.}
\end{enumerate}
\end{defn}

\begin{rem}\label{rem1.2}
\textnormal{Let $s>0$, $0\le b<2$ and $0<\sigma <\sigma _{s} $. If $\sigma $ is not an even integer, assume that $\left\lceil s\right\rceil -1< \sigma $. If $s<1$, in addition, suppose further that $\sigma \ge 1$. Then one can verify that $f(u)=\lambda \left|u\right|^{\sigma } u$ with $\lambda \in \C$ is a model case of class $C\left(\sigma ,s,b\right)$. See \cite{CFH11, DYC13} for example.}
\end{rem}

The equation (1.1) is equivalent to
\begin{equation} \label{GrindEQ__1_6_}
\;u(t)=S(t)u_{0} -i\lambda \int _{0}^{t}S(t-\tau )|x|^{-b} f\left(u\left(\tau \right)\right)d\tau  ,
\end{equation}
where $S(t)=e^{it\Delta } $ is the Schr\"{o}dinger semi-group.

When $0\le s<\frac{n}{2} $, $\sigma _{s} $ given in (1.2) is said to be a critical power in $H^{s} (\R^{n} )$. And $\sigma \left(<\sigma _{s} \right)$ is known as a subcritical power in $H^{s} (\R^{n} )$ for $s\ge 0$.

\begin{defn}\label{defn 1.3}
\textnormal{We say that a pair $(\gamma (p),\;p)$ is admissible if
\begin{equation} \label{GrindEQ__1_7_}
\left\{\begin{array}{l} {2\le p\le \frac{2n}{n-2} ,\;n\ge 3,} \\ {2\le p<\infty , \;n=2,} \\ {2\le p\le \infty , \;n=1,} \end{array}\right.
\end{equation}
and
\begin{equation} \label{GrindEQ__1_8_}
\frac{2}{\gamma (p)} =\frac{n}{2} -\frac{n}{p} .
\end{equation}
We also define the following Strichartz spaces:
\[S\left(I,\;H^{s} \right)=\left\{u:\, \left\| u\right\| _{S\left(I,\;H^{s} \right)} <\infty \right\}, ~S\left(I,\;\dot{H}^{s} \right)=\left\{u:\, \left\| u\right\| _{S\left(I,\;\dot{H}^{s} \right)} <\infty \right\}\]
and dual Strichartz spaces:
\[S'\left(I,\;H^{s} \right)=\left\{u:\, \left\| u\right\| _{S'\left(I,\;H^{s} \right)} <\infty \right\}, ~S'\left(I,\;\dot{H}^{s} \right)=\left\{u:\, \left\| u\right\| _{S'\left(I,\;\dot{H}^{s} \right)} <\infty \right\},\]
whose norms are defined by
\[\left\| u\right\| _{S\left(I,\;\dot{H}^{s} \right)} ={\mathop{\sup }\limits_{\left(\gamma \left(r\right),\;r\right)\in A}} \left\| u\right\| _{L^{\gamma \left(r\right)} \left(I,\;\dot{H}_{r}^{s} \right)},~\left\| u\right\| _{S\left(I,\;H^{s} \right)} ={\mathop{\sup }\limits_{\left(\gamma \left(r\right),\;r\right)\in A}} \left\| u\right\| _{L^{\gamma \left(r\right)} \left(I,\;H_{r}^{s} \right)} ,\]
\[\left\| u\right\| _{S'\left(I,\;\dot{H}^{s} \right)} ={\mathop{\inf }\limits_{\left(\gamma \left(r\right),\;r\right)\in A}} \left\| u\right\| _{L^{\gamma \left(r\right)^{{'} } } \left(I,\;\dot{H}_{r'}^{s} \right)},~\left\| u\right\| _{S'\left(I,\;H^{s} \right)} ={\mathop{\inf }\limits_{\left(\gamma \left(r\right),\;r\right)\in A}} \left\| u\right\| _{L^{\gamma \left(r\right)^{{'} } } \left(I,\;H_{r'}^{s} \right)} ,\]
where $s\in \R$, $A=\left\{\left(\gamma \left(r\right),\;r\right):\;\left(\gamma \left(r\right),\;r\right)\textrm{is admissible}\right\}$ and $I\subset \R$ is an interval. We also use the notation $S\left(I,\;L^{{2}} \right)$ and $S'\left(I,\;L^{{2}} \right)$ instead of $S\left(I,\;\dot{H}^{0} \right)$ and $S'\left(I,\;\dot{H}^{0} \right)$, respectively.}
\end{defn}
If $b=0$, the equation (1.1) is the well-known classic nonlinear Schr\"{o}dinger equation (NLS) which has been widely studied over the last three decades. See, for example, \cite{C03, CFH11, DYC13, LP15, BHHG11} and the references therein.

In this paper, we consider the case $0<b<2$. The INLS equation (1.1) arises in nonlinear optics and it has also been widely studied by several authors during the last two decades. We refer the reader to \cite{BKMT11, BKVLMT12, G00, LT94} for the physical background of the INLS equation (1.1). We also refer the reader to [1, 4, 5, 8, 10--16, 18] for recent work on the INLS equation (1.1). In particular, the author in \cite{G17} studied the local and global well-posedness in $H^{s} (\R^{n} )$ with $0\le s\le \min \left\{\frac{n}{2} ,\;1\right\}$ for the INLS equation (1.1) with $f(u)=\lambda \left|u\right|^{\sigma } u$ by using the contraction mapping principle based on Strichartz estimates. Recently, the authors in \cite{AK21} improved the local well-posedness result of \cite{G17} by extending the validity of $s$ and $b$. More precisely, they obtained the following existence result (see Theorem 1.4 and Theorem 1.6 in \cite{AK21}).

\begin{thm}[\cite{AK21}]\label{thm 1.4.}
Let $n\in \N$, $0\le s<\min \left\{n,\;\frac{n}{2} +{1}\right\}$, $0<b<\hat{2}$ and $0<\sigma <\sigma _{s}$, where $\sigma _{s}$ is given in $(1.2)$ and
\begin{equation} \label{GrindEQ__1_9_}
\hat{2}=\left\{\begin{array}{l} {\min \left\{2,\; 1+\frac{n-2s}{2} \right\},\;n\ge 3,} \\ {n-s,\;n=1,\;2.} \end{array}\right.
\end{equation}
Assume that $f$ is of class ${\rm X} \left(\sigma,s,b\right)$ (see Definition 1.1 of \cite{AK21}). If $\phi \in H^{s} \left(\R^{n} \right)$, then there exists $T=T\left(\left\| \phi \right\| _{H^{s} } \right)>0$ such that $(1.1)$ has a unique solution satisfying
\begin{equation} \label{GrindEQ__1_10_}
u\in C\left(\left[-T,\;T\right],H^{s} \right)\bigcap L^{\gamma \left(p\right)} \left(\left[-T,\;T\right],\;H_{p}^{s} \left(\R^{n} \right)\right),
\end{equation}
for any admissible pair $(\gamma (p),\;p)$.
\end{thm}
The existence of solutions in Theorem 1.4 is established by using contraction mapping argument based on Strichartz estimates. Repeating this argument, we can extend the above solution to the maximal solution of (1.1), that is, we have the following existence result.
\begin{thm}\label{thm 1.5.}
Let $n\in \N$, $0\le s<\min \left\{n,\;\frac{n}{2} +{1}\right\}$, $0<b<\hat{2}$ and $0<\sigma <\sigma _{s} $. Assume that $f$ is of class ${\rm X} \left(\sigma,s,b\right)$. Then for any $\phi \in H^{s} \left(\R^{n} \right)$, there exist $T_{\max } =T_{\max } \left(\left\| \phi \right\| _{H^{s} } \right)>0$, $T_{\min } =T_{\min } \left(\left\| \phi \right\| _{H^{s} } \right)>0$ and a unique, maximal solution of $(1.1)$ satisfying
\begin{equation} \label{GrindEQ__1_11_}
C\left(\left(-T_{\min } ,\;T_{\max } \right),H^{s} \right)\bigcap L_{loc}^{\gamma (p)} \left(\left(-T_{\min } ,\;T_{\max } \right),\;H_{p}^{s} \left(\R^{n} \right)\right)
\end{equation}
for any admissible pair $(\gamma (p),\;p)$.
\end{thm}

\begin{rem}\label{rem 1.6.}
\textnormal{Obviously, if $f$ is of class $C\left(\sigma ,s,b\right)$, then $f$ is of class ${\rm X} \left(\sigma,s,b\right)$}
\end{rem}

As in the study of the classic nonlinear Schr\"{o}dinger (NLS) equation, it is natural to ask whether the above solution depends continuously on the initial data. See \cite{C03, CFH11, DYC13} for example. Since Theorem 1.4 and Theorem 1.5 are proved by using the fixed point argument, one can conjecture that the dependence of solution on the initial data is locally Lipschitz. However, the metric space in which one applies Banach's fixed point theorem involves Sobolev norms of order $s$, while the distance only involves Lebesgue norms. Thus the flow is locally Lipschitz for Lebesgue norms, and the continuous dependence in the following sense follows by interpolation inequalities.
\begin{thm}\label{thm 1.7.}
Let $n\in \N$, $0<s<\min \left\{n,\;1+\frac{n}{2} \right\}$, $0<b<\hat{2}$ and $0<\sigma <\sigma _{s} $. Assume that $f$ is of class ${\rm X} \left(\sigma,s,b\right)$. Then for any given $\phi \in H^{s} \left(\R^{n} \right)$, the corresponding solution $u$ of the INLS equation $(1.1)$ in Theorem 1.5 depends continuously on the initial data $\phi $ in the following sense. There exists $0<T<T_{\max } ,\, T_{\min } $ such that if $\phi _{m} \to \phi $ in $H^{s} \left(\R^{n} \right)$ and if $u_{m} $ denotes the solution of $(1.1)$ with the initial data $\phi _{m} $, then $0<T<T_{\max } \left(\phi _{m} \right),\, T_{\min } \left(\phi _{m} \right)$ for all sufficiently large $m$ and $u_{m} $ is bounded in $S\left(\left[-T,\;T\right],\;H^{s} \right)$. Moreover, $u_{m} \to u$ in $L^{\gamma (p)} \left(\left[-T,\;T\right],\;H_{p}^{s-\varepsilon } \left(\R^{n} \right)\right)$ as $m\to \infty $ for all $\varepsilon >0$ and all admissible pair $\left(\gamma \left(p\right),\, p\right)$. In particular, $u_{m} \to u$ in $C\left(\left[-T,\;T\right],\; H^{s-\varepsilon } \right)$ for all $\varepsilon >0$.
\end{thm}
\begin{proof}
We can easily prove Theorem 1.7 by combining the argument used in \cite{AK21} with that used in the proof of Theorem 4.9.1 in \cite{C03} and we omit the details.
\end{proof}
But the above result is weaker than what would be ``standard", i.e. $\varepsilon =0$ (see e.g. \cite{C03}). In the study of the classic NLS equation (which corresponds to (1.1) with $b=0$), the continuous dependence in the standard sense in $H^{s}$ was investigated by \cite{CFH11, DYC13}, where the Besov space theory was mainly applied. As in the study of the classic NLS equation, we wonder if the solution of the INLS equation (1.1) depends continuously on the initial data in the standard sense in $H^{s}$, i.e. in the sense that the local solution flow is continuous $H^{s} \to H^{s} $. However, up to the knowledge of the authors, the continuous dependence of (1.1) in the standard sense is not yet clear. On the other hand, as we can see in the previous work on (1.1), it seems difficult to apply the Besov space theory in the study of INLS equation (1.1), due to the term $|x|^{-b}$.

Taking these consideration into our account, in this paper, we will investigate the continuous dependence of the Cauchy problem for the INLS equation (1.1) in the standard sense in $H^{s}$. To arrive at this goal, we first establish the various estimates of the term $f(u)-f(v)$ in the Sobolev spaces of order $s$ (see Lemma 3.1 --Lemma 3.4, Lemma 3.8--Lemma 3.11). And then by using them, we establish the estimates of the term $|x|^{-b}(f(u)-f(v))$ (see Lemma 3.6, Lemma 3.7, Lemma 3.12--Lemma 3.15). These estimates play a crucial role in proving the continuous dependence of the Cauchy problem for the INLS equation (1.1). The main result of this paper is the following.

\begin{thm}\label{thm 1.8.}
Let $n\in \N$, $0<s<\min \left\{n,\;1+\frac{n}{2} \right\}$, $0<b<\hat{2}$ and $0<\sigma <\sigma _{s} $. Assume that $f$ is of class $C\left(\sigma ,s,b\right)$. Then for any given $\phi \in H^{s} \left(\R^{n} \right)$, the corresponding solution $u$ of the INLS equation $(1.1)$ in Theorem 1.5 depends continuously on the initial data $\phi $ in the following sense. For any interval $\left[-S,\, T\right]\subset \left(-T_{\min } \left(\phi \right),\;T_{\max } \left(\phi \right)\right)$, and every admissible pair $\left(\gamma \left(p\right),\;p\right)$, if $\phi _{m} \to \phi $ in $H^{s} \left(\R^{n} \right)$ and if $u_{m} $ denotes the solution of $(1.1)$ with the initial data $\phi _{m} $, then $u_{m} \to u$ in $L^{\gamma (p)} (\left[-S,\;T\right],\;H_{p}^{s} (\R^{n} ))$ as $m\to \infty $. In particular, $u_{m} \to u$ in $C\left(\left[-S,\;T\right],\;H^{s} \right)$.
In addition, if $f(u)$ is a polynomial in $u$ and $\bar{u}$, or if $f(u)$ is not a polynomial and $\sigma \ge \left\lceil s\right\rceil $, then the dependence is locally Lipschitz.
\end{thm}
By Remark 1.2, Theorem 1.8 applies in particular to the model case $f(u)=\lambda \left|u\right|^{\sigma } u$ with $\lambda \in \C$.

\begin{cor}\label{cor 1.9.}
Assume $n\in \N$, $0<s<\min \left\{n,\;1+\frac{n}{2} \right\}$, $0<b<\hat{2}$ and $0<\sigma <\sigma _{s} $. Let $f(u)=\lambda \left|u\right|^{\sigma } u$ with $\lambda \in \C$. If $\sigma $ is not an even integer, assume further $\left\lceil s\right\rceil -1< \sigma $. If $s<1$, in addition, suppose further that $\sigma \ge 1$. Then for any given $\phi \in H^{s} \left(\R^{n} \right)$, the corresponding solution $u$ of the INLS equation $(1.1)$ in Theorem 1.5 depends continuously on the initial data $\phi $ in the sense of Theorem 1.8.
\end{cor}

\begin{rem}\label{rem 1.10.}
\textnormal{When $0<s<1$, Corollary 1.9 only covers the case $1\le \sigma <\frac{4-2b}{n-2s}$, while the existence result holds for $0<\sigma <\frac{4-2b}{n-2s}$. It is an interesting problem to show the continuous dependence in the standard sense when $0<s<1$ and $0<\sigma<1$.}
\end{rem}

This paper is organized as follows. In Section 2, we introduce some basic notation and give some preliminary results. In Section 3, we prove Theorem 1.8

\noindent
\section{Preliminaries}

We begin by introducing some basic notation. $F$ denotes the Fourier transform; $F^{-1} $ denotes the inverse Fourier transform. We denote by $p'$ the dual number of $p\in [1,\;\infty ]$, i.e. $1/p+1/p'=1$. For $s\in \R$, we denote by $\left[s\right]$ the largest integer which is less than or equals to $s$ and by $\left\lceil s\right\rceil$ the minimal integer which is larger than or equals to $s$. $C\left(>0\right)$ will denote a positive universal constant, which can be different at different places. $a\lesssim b$ means $a\le Cb$ for some constant $C>0$. For $B \subset \R^{n}$, $\chi_{B}$ denotes the characteristic function of $B$, i.e. $\chi_{B}(x)=1$ for $x\in B$, and $\chi_{B}(x)=0$ for $x \in B^{C}$.
For a multi-index $\alpha =\left(\alpha _{1} ,\;\alpha _{2} \;,\;\ldots ,\;\alpha _{n} \right)$, denote
\[D^{\alpha } =\partial _{x_{1} }^{\alpha _{1} } \cdots \partial _{x_{n} }^{\alpha _{n} },~ \left|\alpha \right|=\left|\alpha _{1} \right|+\cdots +\left|\alpha _{n} \right|.\]
For a function $f(z)$ of a complex variable $z$ and for a positive integer $k$, $k$-th order complex derivative of $f(z)$ is defined by
\[f^{\left(k\right)} (z)=\left(\frac{\partial ^{k} f}{\partial z^{k} } ,\;\frac{\partial ^{k} f}{\partial z^{k-1} \partial \bar{z}} ,\;\ldots ,~\frac{\partial ^{k} f}{\partial z\partial \bar{z}^{k-1} } ,~\frac{\partial ^{k} f}{\partial \bar{z}^{k} } \right),\]
where
\[\frac{\partial f}{\partial z} =\frac{1}{2} \left(\frac{\partial f}{\partial x} -i\frac{\partial f}{\partial y} \right),~ \frac{\partial f}{\partial \bar{z}} =\frac{1}{2} \left(\frac{\partial f}{\partial x} +i\frac{\partial f}{\partial y} \right).\]
We also define its norm as
\[\left|f^{\left(k\right)} (z)\right|=\sum _{i=0}^{k}\left|\frac{\partial ^{k} f}{\partial z^{k-i} \partial \bar{z}^{i} } \right| .\]

As usual, $L^{p} (\R^{n} )$ is the Lebesgue space,
\[\left\| f\right\| _{p} =\;\left\| f\right\| _{L^{p} (\R^{n} )} =\left(\int _{\R^{n} }\left|f(x)\right|^{p} dx \right)^{1/p} ,\]
with the usual modification when $p=\infty$.
For $s\in \R$ and $1<p<\infty $, the norms of nonhomogeneous Sobolev space $H_{p}^{s} (\R^{n} )$ and homogeneous Sobolev space $\dot{H}_{p}^{s} (\R^{n} )$ are defined by
\[\left\| f\right\| _{H_{p}^{s} (\R^{n} )} =\left\| F^{-1} \left(1+\left|\xi \right|^{2} \right)^{\frac{s}{2} } Ff\right\| _{L^{p} (\R^{n} )} , ~\left\| f\right\| _{\dot{H}_{p}^{s} (\R^{n} )} =\left\| F^{-1} \left|\xi \right|^{s} Ff\right\| _{L^{p} (\R^{n} )} .\]
We shall abbreviate $H_{2}^{s} (\R^{n} )$ and $\dot{H}_{2}^{s} (\R^{n} )$ respectively as $H^{s} (\R^{n} )$ and $\dot{H}^{s} (\R^{n} )$. We will also use the space-time mixed space $L^{\gamma } (I,\;X(\R^{n} ))$ whose norm is defined by
\[\;\left\|f\right\|_{L^{\gamma } (I,\;X(\R^{n} ))} =\left(\int _{I}\left\| f\right\| _{X(\R^{n} )}^{\gamma } dt \right)^{\frac{1}{\gamma } } ,\]
with the usual modification when $\gamma=\infty$, where $I\subset \R$ is an interval and $X(\R^{n} )$ is a normed space on $\R^{n} $ such as Lebesgue space or Sobolev space. If there is no confusion, $\R^{n} $ will be omitted in various function spaces. For two normed spaces $X$ and $Y$, $X\subset Y$ means that the space $X$ is continuously embedded in the space $Y$, that is, there exists a constant $C\left(>0\right)$ such that $\left\| f\right\| _{Y} \le C\left\| f\right\| _{X} $ for all $f\in X$.

Next, we recall some preliminary results.

\begin{lem}[Fractional Product Rule, \cite{CW91}]\label{lem 2.1.}
Let $s\ge 0$, $1<r,\;r_{2} ,\;p_{1} <\infty $, $1<r_{1} ,\;p_{2} \le \infty $. Assume that
\[\frac{1}{r} =\frac{1}{r_{i} }+\frac{1}{p_{i} }~(i=1,\;2).\]
Then we have
\begin{equation} \label{GrindEQ__2_1_}
\left\| fg\right\| _{\dot{H}_{r}^{s} } \lesssim\left\| f\right\| _{r_{1} } \left\| g\right\| _{\dot{H}_{p_{1} }^{s} } +\left\| f\right\| _{\dot{H}_{r_{2} }^{s} } \left\| g\right\| _{p_{2} } .
\end{equation}
\end{lem}

\begin{cor}\label{cor 2.2.}
Let $s\ge 0$, $q\in \N$. Let ${1}<r,\;r_{k}^{i} <\infty $ for $1\le i,\;k\le q$. Assume that
\[\frac{1}{r} =\sum _{i=1}^{q}\frac{1}{r_{k}^{i} }  , \]
for any $1\le k\le q$. Then we have
\begin{equation} \label{GrindEQ__2_2_}
\left\| \prod _{i=1}^{q}f_{i}  \right\| _{\dot{H}_{r}^{s} } \lesssim\sum _{k=1}^{q}\left(\left\| f_{k} \right\| _{\dot{H}_{r_{k}^{k} }^{s} } \prod _{i\in I_{k} }\left\| f_{i} \right\| _{r_{k}^{i} }  \right) ,
\end{equation}
where $I_{k} =\left\{i\in \N:\;1\le i\le q,\;i\ne k\right\}$.
\end{cor}
\begin{proof} Using Lemma 2.1, H\"{o}lder inequality and induction on $q$, we can easily prove (2.2).\end{proof}

\begin{lem}[Fractional Chain Rule]\label{lem 2.3.}
Suppose $G\in C^{1} (\C)$ and $s\in \left(0,\;1\right)$. Then for $1<r,\;r_{2} <\infty $, and $1<r_{1} \le \infty $ satisfying $\frac{1}{r} =\frac{1}{r_{1} } +\frac{1}{r_{2} } $,
\begin{equation} \label{GrindEQ__2_3_}
\left\| G\left(u\right)\right\| _{\dot{H}_{r}^{s} } \lesssim\left\| G'\left(u\right)\right\| _{r_{1} } \left\|u\right\|_{\dot{H}_{r_{2} }^{s} } .
\end{equation}
\end{lem}

\begin{proof} See \cite{CW91} for \textit{$1<q_{1} <\infty $ } and \cite{KPV93} for \textit{$q_{1} =\infty $}.\end{proof}

Next, we recall some useful embeddings on Sobolev spaces. See \cite{BHHG11} for example.

\begin{lem}\label{lem 2.4.}
Let $-\infty <s_{2} \le s_{1} <\infty $ and $1<p_{1} \le p_{2} <\infty $ with $s_{1} -\frac{n}{p_{1} } =s_{2} -\frac{n}{p_{2} } $. Then we have the following embeddings:
\[\dot{H}_{p_{1} }^{s_{1} } \subset \dot{H}_{p_{2} }^{s_{2}},~H_{p_{1} }^{s_{1} } \subset H_{p_{2} }^{s_{2} } .\]
\end{lem}
\begin{lem}\label{lem 2.5.}
Let $-\infty <s<\infty $ and $1<p<\infty $. Then we have
\begin{enumerate}
\item  $H_{p}^{s+\varepsilon } \subset H_{p}^{s}~(\varepsilon >0)$,

\item  $H_{p}^{s} \subset L^{\infty }~(s>{n \mathord{\left/{\vphantom{n p}}\right.\kern-\nulldelimiterspace} p} )$,

\item  $H_{p}^{s} =L^{p} \bigcap \dot{H}_{p}^{s}~ (s>0)$.
\end{enumerate}
\end{lem}

\begin{cor}\label{cor 2.6.}
Let $-\infty <s_{2} \le s_{1} <\infty $ and $1<p_{1} \le p_{2} <\infty $ with $s_{1} -\frac{n}{p_{1} } \ge s_{2} -\frac{n}{p_{2} } $. Then we have $H_{p_{1} }^{s_{1} } \subset H_{p_{2} }^{s_{2} } $.
\end{cor}
\begin{proof}
The proof follows from Lemma 2.4 and Lemma 2.5 (1).
\end{proof}

\begin{lem}[\cite{AK21}]\label{lem 2.7.}
Let $s>0$, $1<p<\infty $ and $v=s-\left[s\right]$. Then $\sum _{\left|\alpha \right|=\left[s\right]}\left\| D^{\alpha } f\right\| _{\dot{H}_{p}^{v} }  $ is an equivalent norm on $\dot{H}_{p}^{s} \left(\R^{n} \right)$.
\end{lem}

We end this section with recalling the well-known Strichartz estimates. See \cite{C03, LP15, BHHG11} for instance.

\begin{lem}[Strichartz estimates]\label{lem 2.8.}
Let $S(t)=e^{it\Delta } $ and $s\in \R$. Then we have
\begin{equation} \label{GrindEQ__2_4_}
\left\| S(t)\phi \right\| _{S\left(I,\;\dot{H}^{s} \right)} \lesssim\left\| \phi \right\| _{\dot{H}^{s} } ,
\end{equation}
\begin{equation} \label{GrindEQ__2_5_}
\left\| \int _{0}^{t}S(t-\tau )f(\tau )d\tau  \right\| _{S\left(I,\;\dot{H}^{s} \right)} \lesssim\left\| f\right\| _{S'\left(I,\;\dot{H}^{s} \right)} ,
\end{equation}
where $I$ is an arbitrary interval satisfying $0\in \bar{I}$.
\end{lem}

\noindent
\section{Proof of Theorem 1.7 }

In this section, we prove Theorem 1.7.
It follows from Theorem 1.5 that given $\phi \in H^{s} \left(\R^{n} \right)$, there exist $T_{\max } ,\, T_{\min } \in (0,\, \infty ]$ such that (1.1) has a unique, maximal solution $u\in C\left(\left(-T_{\min } ,\;T_{\max } \right),H^{s} \left(\R^{n} \right)\right)$. There exists $0<T<T_{\max } ,\, T_{\min } $ such that if $\phi _{m} \to \phi $ in $H^{s} \left(\R^{n} \right)$ and if $u_{m} $ denotes the solution of (1.1) with the initial data $\phi _{m} $, then $\left\| \phi _{m} \right\| _{H^{s} } \le 2\left\| \phi \right\| _{H^{s} } $ for $n$ large, and we have $0<T<T_{\max } \left(\phi _{m} \right),\, T_{\min } \left(\phi _{m} \right)$ for all sufficiently large $m$ and $u_{m} $ is bounded in $S\left(\left[-T,\;T\right],\;H^{s} \right)$.
Since $u$, $u_{m} $ satisfy integral equation:
\[u\left(t\right)=e^{it\Delta } \phi -i\lambda \int _{0}^{t}e^{i\left(t-\tau \right)\Delta } |x|^{-b} f\left(u\left(\tau \right)\right)d\tau  ,\]
\[u_{m} \left(t\right)=e^{it\Delta } \phi _{m} -i\lambda \int _{0}^{t}e^{i\left(t-\tau \right)\Delta } |x|^{-b} f\left(u_{m} \left(\tau \right)\right)d\tau  ,\]
respectively, we have
\begin{equation} \label{GrindEQ__3_1_}
u_{m} \left(t\right)-u\left(t\right)=e^{it\Delta } \left(\phi _{m} -\phi \right)-i\lambda \int _{0}^{t}e^{i\left(t-\tau \right)\Delta } |x|^{-b} \left(f\left(u_{m} \left(\tau \right)\right)-f\left(u\left(\tau \right)\right)\right)d\tau
\end{equation}

We are to prove that there exits $T>0$ sufficiently small such that as $m\to \infty $,
\begin{equation} \label{GrindEQ__3_2_}
u_{m} \to u ~~\textrm{in}~~ L^{\gamma \left(r\right)} \left(\left[-T,\, T\right],\, H_{r}^{s} \left(\R^{n} \right)\right),
\end{equation}
for every admissible pair $\left(r\left(r\right),\, r\right)$.
If this has been done, then Theorem 1.7 follows by iterating this property to cover any compact subset of $\left(-T_{\min } ,\, T_{\max } \right)$. See \cite{CFH11, DYC13} and Chapter 3 or 4 of \cite{C03}.

We divide the proof of (3.2) in two cases: $0<s<\frac{n}{2} $ and $\frac{n}{2} \le s<\min \left\{\frac{n}{2} +1,\, n\right\}$.

\subsection{Case 1. $0<s<\frac{n}{2} $}

\begin{lem}\label{lem 3.1.}
Let $p>1$, $0<s<1$ and $\sigma \ge 1$. Assume that $f\in C^{2} \left(\C\to \C\right)$ satisfies
\begin{equation} \label{GrindEQ__3_3_}
\left|f^{\left(k\right)} \left(u\right)\right|\lesssim\left|u\right|^{\sigma +1-k} ,
\end{equation}
for any $0\le k\le 2$ and $u\in \C$. Suppose also that
\begin{equation} \label{GrindEQ__3_4_}
\frac{1}{p} =\sigma \left(\frac{1}{r} -\frac{s}{n} \right)+\frac{1}{r} ,~\frac{1}{r} -\frac{s}{n} >0.
\end{equation}
Then we have
\[\left\| f(u)-f(v)\right\| _{\dot{H}_{p}^{s} } \lesssim\left(\left\| u\right\| _{\dot{H}_{r}^{s} }^{\sigma } +\left\| v\right\| _{\dot{H}_{r}^{s} }^{\sigma } \right)\left\| u-v\right\| _{\dot{H}_{r}^{s} } .\]
\end{lem}
\begin{proof}
We can see that
\[f(u)-f(v)=\int _{0}^{1}\left[f_{z} \left(v+t\left(u-v\right)\right)\left(u-v\right)+f_{\bar{z}} \left(v+t\left(u-v\right)\right)\left(\bar{u}-\bar{v}\right)\right]dt .\]
Without loss of generality and for simplicity, we assume that $f$ is a function of a real variable, i.e. we have
\[f(u)-f(v)=\left(u-v\right)\int _{0}^{1}f'\left(v+t\left(u-v\right)\right)dt .\]
Putting
\begin{equation} \label{GrindEQ__3_5_}
\frac{1}{p_{1} } =\frac{\sigma }{a} ,\, \frac{1}{p_{2} } =\frac{1}{p_{1} } +\frac{s}{n} ,\, \frac{1}{a} =\frac{1}{r} -\frac{s}{n},
\end{equation}
it follows from (3.4) and Lemma 2.1 (fractional product rule) that
\begin{equation}\nonumber
\left\| f(u)-f(v)\right\| _{\dot{H}_{p}^{s} }=\left\| \left(u-v\right)\int _{0}^{1}f'\left(v+t\left(u-v\right)\right)dt \right\| _{\dot{H}_{p}^{s} }\lesssim I_{1} +I_{2},
\end{equation}
where

$$
I_{1}=\left\| \int _{0}^{1}f'\left(v+t\left(u-v\right)\right)dt \right\| _{p_{1} } \left\| u-v\right\| _{\dot{H}_{r}^{s} },
$$

$$
I_{2}=\left\| \int _{0}^{1}f'\left(v+t\left(u-v\right)\right)dt \right\| _{\dot{H}_{p_{2} }^{s} } \left\| u-v\right\| _{a}.
$$
First, we estimate $I_{1}$. We have
\begin{equation} \label{GrindEQ__3_6_}
\left\| \int _{0}^{1}f'\left(v+t\left(u-v\right)\right)dt \right\| _{p_{1} } \le \int _{0}^{1}\left\| f'\left(v+t\left(u-v\right)\right)\right\| _{p_{1} } dt .
\end{equation}
It also follows from (3.3) that
\begin{equation} \label{GrindEQ__3_7_}
\left|f'\left(v+t\left(u-v\right)\right)\right|\lesssim{\mathop{\max }\limits_{t\in \left[0,\, 1\right]}} \left|v+t\left(u-v\right)\right|^{\sigma } \lesssim \left|u\right|^{\sigma } +\left|v\right|^{\sigma },
\end{equation}
for any $0\le t \le 1$. (3.5)--(3.7) imply that
\begin{eqnarray}\begin{split}\label{GrindEQ__3_8_}
\left\| \int _{0}^{1}f'\left(v+t\left(u-v\right)\right)dt \right\| _{p_{1} } &\lesssim\left\| \left|u\right|^{\sigma } +\left|v\right|^{\sigma } \right\| _{p_{1} } \le \left\| \left|u\right|^{\sigma } \right\| _{p_{1} } +\left\| \left|v\right|^{\sigma } \right\| _{p_{1} }  \\
&=\left\| u\right\| _{a}^{\sigma } +\left\| v\right\| _{a}^{\sigma } \;\lesssim\left\| u\right\| _{\dot{H}_{r}^{s} }^{\sigma } +\left\| v\right\| _{\dot{H}_{r}^{s} }^{\sigma } ,
\end{split}\end{eqnarray}
where the last inequality follows from the embedding $\dot{H}_{r}^{s} \subset L^{a} $. Hence we have
\begin{equation} \label{GrindEQ__3_9_}
I_{1} \lesssim\left(\left\| u\right\| _{\dot{H}_{r}^{s} }^{\sigma } +\left\| v\right\| _{\dot{H}_{r}^{s} }^{\sigma } \right)\left\| u-v\right\| _{\dot{H}_{r}^{s} } .
\end{equation}
Next, we estimate $I_{2}$. We have
\begin{equation} \label{GrindEQ__3_10_}
\left\| \int _{0}^{1}f'\left(v+t\left(u-v\right)\right)dt \right\| _{\dot{H}_{p_{2} }^{s} } \le \int _{0}^{1}\left\| f'\left(v+t\left(u-v\right)\right)\right\| _{\dot{H}_{p_{2} }^{s} } dt .
\end{equation}
On the other hand, it follows from Lemma 2.3 (fractional chain rule) and (3.3) that
\begin{eqnarray}\begin{split} \label{GrindEQ__3_11_}
\left\| f'\left(v+t\left(u-v\right)\right)\right\| _{\dot{H}_{p_{2} }^{s} } &\lesssim\left\| f''\left(v+t\left(u-v\right)\right)\right\| _{p_{3} } \left\| v+t\left(u-v\right)\right\| _{\dot{H}_{r}^{s} } \\
&\lesssim\left\| \left|v+t\left(u-v\right)\right|^{\sigma -1} \right\| _{p_{3} } \left\| v+t\left(u-v\right)\right\| _{\dot{H}_{r}^{s} } \\
&\lesssim\left\| v+t\left(u-v\right)\right\| _{\dot{H}_{r}^{s} }^{\sigma } ,
\end{split}\end{eqnarray}
where $\frac{1}{p_{3} } =\left(\sigma -1\right)\left(\frac{1}{r} -\frac{s}{n} \right)$, if $\sigma >1$ and $p_{3} =\infty $, if $\sigma =1$. We can also see that
\begin{equation} \label{GrindEQ__3_12_}
\left\| v+t\left(u-v\right)\right\| _{\dot{H}_{r}^{s} } \le {\mathop{\max }\limits_{t\in \left[0,\, 1\right]}} \left\| I^{s} v+t\left(I^{s} u-I^{s} v\right)\right\| _{L^{r} } \lesssim\left\| u\right\| _{\dot{H}_{r}^{s} } +\left\| v\right\| _{\dot{H}_{r}^{s} } .
\end{equation}
(3.10)--(3.12) imply that
\begin{equation} \label{GrindEQ__3_13_}
I_{2} =\left\| \int _{0}^{1}f'\left(v+t\left(u-v\right)\right)dt \right\| _{\dot{H}_{p_{2} }^{s} } \left\| u-v\right\| _{a} \lesssim\left(\left\| u\right\| _{\dot{H}_{r}^{s} }^{\sigma } +\left\| v\right\| _{\dot{H}_{r}^{s} }^{\sigma } \right)\left\| u-v\right\| _{\dot{H}_{r}^{s} } .
\end{equation}
In view of (3.9) and (3.13), we have
\begin{equation} \label{GrindEQ__3_14_}
\left\| f(u)-f(v)\right\| _{\dot{H}_{p}^{s} } \lesssim I_{1} +I_{2} \lesssim\left(\left\| u\right\| _{\dot{H}_{r}^{s} }^{\sigma } +\left\| v\right\| _{\dot{H}_{r}^{s} }^{\sigma } \right)\left\| u-v\right\| _{\dot{H}_{r}^{s} } ,
\end{equation}
this completes the proof.
\end{proof}

\begin{lem}\label{lem 3.2.}
Let $p>1$, $s\ge 1$ and $\sigma >\left\lceil s\right\rceil -1$. Assume that $f\in C^{\left\lceil s\right\rceil } \left(\C\to \C\right)$ satisfies
\begin{equation} \label{GrindEQ__3_15_}
\left|f^{\left(k\right)} \left(u\right)\right|\lesssim\left|u\right|^{\sigma +1-k} ,
\end{equation}
for any $0\le k\le \left\lceil s\right\rceil $ and $u\in \C$. Assume further
\begin{equation} \label{GrindEQ__3_16_}
\left|f^{\left(\left\lceil s\right\rceil\right)} \left(u\right)-f^{\left(\left\lceil s\right\rceil \right)} \left(v\right)\right|\lesssim\left|u-v\right|^{\min \{ \sigma -\left\lceil s\right\rceil +1,\;1\} } \left(\left|u\right|+\left|v\right|\right)^{\max \{ 0,\;\sigma -\left\lceil s\right\rceil \} } ,
\end{equation}
for any $u,\;v\in \C$. Suppose also that \textnormal{(3.4)} holds. Then we have
\begin{eqnarray}\begin{split} \label{GrindEQ__3_17_}
\left\| f(u)-f(v)\right\| _{\dot{H}_{p}^{s} }&\lesssim\left\| u-v\right\| _{a}^{\min \{ \sigma -\left\lceil s\right\rceil +1,\;1\} } \left(\left\| u\right\| _{\dot{H}_{r}^{s} }^{\max \{ \left\lceil s\right\rceil ,\;\sigma \} } +\left\| v\right\| _{\dot{H}_{r}^{s} }^{\max \{ \left\lceil s\right\rceil ,\;\sigma \} } \right)\\
&~~+\left(\left\| u\right\| _{\dot{H}_{r}^{s} }^{\sigma } +\left\| v\right\| _{\dot{H}_{r}^{s} }^{\sigma } \right)\left\| u-v\right\| _{\dot{H}_{r}^{s} } ,
\end{split}\end{eqnarray}
where $a=\frac{rn}{n-rs} $. Moreover, if $\sigma \ge \left\lceil s\right\rceil $, then we have
\begin{equation} \label{GrindEQ__3_18_}
\left\| f(u)-f(v)\right\| _{\dot{H}_{p}^{s} } \lesssim\left(\left\| u\right\| _{\dot{H}_{r}^{s} }^{\sigma } +\left\| v\right\| _{\dot{H}_{r}^{s} }^{\sigma } \right)\left\| u-v\right\| _{\dot{H}_{r}^{s} } .
\end{equation}
\end{lem}
\begin{proof}
By Lemma 2.7, we have
\begin{equation} \label{GrindEQ__3_19_}
\left\| f(u)-f(v)\right\| _{\dot{H}_{p}^{s} } \lesssim\sum _{\left|\alpha \right|=\left[s\right]}\left\| D^{\alpha } f(u)-D^{\alpha } f(v)\right\| _{\dot{H}_{p}^{v} }  ,
\end{equation}
where $v=s-\left[s\right]$. Without loss of generality and for simplicity, we assume that $f$ is a function of a real variable. It follows from the Leibniz rule of derivatives that
\[D^{\alpha } f(u)=\sum _{q=1}^{\left|\alpha \right|}\sum _{\Lambda _{\alpha }^{q} }C_{\alpha ,\;q} f^{\left(q\right)} \left(u\right)\prod _{i=1}^{q}D^{\alpha _{i} } u   \]
where $\Lambda _{\alpha }^{q} =\left(\alpha _{1} +\cdots +\alpha _{q} =\alpha ,\;\left|\alpha _{i} \right|\ge 1\right)$. Thus we have
\begin{eqnarray}\begin{split} \label{GrindEQ__3_20_}
\left\| D^{\alpha } f(u)-D^{\alpha } f(v)\right\| _{\dot{H}_{p}^{v} } &=\left\| \sum _{q=1}^{\left|\alpha \right|}\sum _{\Lambda _{\alpha }^{q} }C_{\alpha ,\;q} \left(f^{\left(q\right)} \left(u\right)\prod _{i=1}^{q}D^{\alpha _{i} } u -f^{\left(q\right)} \left(v\right)\prod _{i=1}^{q}D^{\alpha _{i} } v \right)  \right\| _{\dot{H}_{p}^{v} } \\
&\le \sum _{q=1}^{\left|\alpha \right|}\sum _{\Lambda _{\alpha }^{q} }C_{\alpha ,\;q} \left\| f^{\left(q\right)} \left(u\right)\prod _{i=1}^{q}D^{\alpha _{i} } u -f^{\left(q\right)} \left(v\right)\prod _{i=1}^{q}D^{\alpha _{i} } v \right\| _{\dot{H}_{p}^{v} }    \\
&\le \sum _{q=1}^{\left|\alpha \right|}\sum _{\Lambda _{\alpha }^{q} }C_{\alpha ,\;q} \left(\left\| I_{1} \right\| _{\dot{H}_{p}^{v} } +\left\| II_{1} \right\| _{\dot{H}_{p}^{v} } \right) ,
\end{split}\end{eqnarray}
where
\begin{equation} \label{GrindEQ__3_21_}
I_{1} =\left(f^{\left(q\right)} \left(u\right)-f^{\left(q\right)} \left(v\right)\right)\prod _{i=1}^{q}D^{\alpha _{i} } u ,~II_{1} =f^{\left(q\right)} \left(v\right)\left(\prod _{i=1}^{q}D^{\alpha _{i} } u -\prod _{i=1}^{q}D^{\alpha _{i} } v \right).
\end{equation}
On the other hand, it follows from (3.4) that
\begin{equation} \label{GrindEQ__3_22_}
\frac{1}{p} =\frac{\sigma +1-q}{a} +\sum _{i=1}^{q}\frac{1}{a_{i} }  +\frac{v}{n} .
\end{equation}
where
\begin{equation} \label{GrindEQ__3_23_}
\frac{1}{a} =\frac{1}{r} -\frac{s}{n} ,~\frac{1}{a_{i} } =\frac{1}{r} -\frac{s-\left|\alpha _{i} \right|}{n} .
\end{equation}
Using (3.23) and Lemma 2.4, we have the embeddings: $\dot{H}_{r}^{s} \subset L^{a} $ and $\dot{H}_{r}^{s} \subset \dot{H}_{a_{i} }^{\left|\alpha _{i} \right|} $.

We divide the study in two cases: $s\in \N$ and $s\notin \N$.

\textbf{Case 1.} $s\in \N$, i.e. $v=0$.

\emph{Step 1.1}. First, we estimate $\left\| I_{1} \right\| _{p} $, where $I_{1}$ is given in (3.21). Using H\"{o}lder inequality and (3.22), we have
\begin{eqnarray}\begin{split} \label{GrindEQ__3_24_}
\left\| I_{1} \right\| _{p} &=\left\| \left(f^{\left(q\right)} \left(u\right)-f^{\left(q\right)} \left(v\right)\right)\prod _{i=1}^{q}D^{\alpha _{i} } u \right\| _{p} \le \left\| f^{\left(q\right)} \left(u\right)-f^{\left(q\right)} \left(v\right)\right\| _{p_{5} } \prod _{i=1}^{q}\left\| D^{\alpha _{i} } u\right\| _{a_{i} }  \\
&\lesssim \left\| f^{\left(q\right)} \left(u\right)-f^{\left(q\right)} \left(v\right)\right\| _{p_{5} } \prod _{i=1}^{q}\left\| u\right\| _{\dot{H}_{a_{i} }^{\left|\alpha _{i} \right|} }  \lesssim \left\| f^{\left(q\right)} \left(u\right)-f^{\left(q\right)} \left(v\right)\right\| _{p_{5} } \left\| u\right\| _{\dot{H}_{r}^{s} }^{q},
\end{split}\end{eqnarray}
where $\frac{1}{p_{5} } =\frac{\sigma +1-q}{a} $ and the last inequality follows from the embedding $\dot{H}_{r}^{s} \subset \dot{H}_{a_{i} }^{\left|\alpha _{i} \right|} $.

$\cdot$ If $q=\left|\alpha \right|=\left\lceil s\right\rceil $, then it follows from (3.16) and H\"{o}lder inequality that
\begin{eqnarray}\begin{split} \label{GrindEQ__3_25_}
\left\| f^{\left(\left\lceil s\right\rceil \right)} \left(u\right)-f^{\left(\left\lceil s\right\rceil \right)} \left(v\right)\right\| _{p_{5} } &\lesssim\left\| \left|u-v\right|^{\min \{ \sigma -\left\lceil s\right\rceil +1,\;1\} } \left(\left|u\right|+\left|v\right|\right)^{\max \{ 0,\;\sigma -\left\lceil s\right\rceil \} } \right\| _{p_{5} }  \\
&\lesssim\left\| u-v\right\| _{a}^{\min \{ \sigma -\left\lceil s\right\rceil +1,\;1\} } \left\| \left|u\right|+\left|v\right|\right\| _{a}^{\max \{ 0,\;\sigma -\left\lceil s\right\rceil \} }  \\
&\lesssim\left\| u-v\right\| _{a}^{\min \{ \sigma -\left\lceil s\right\rceil +1,\;1\} } \left(\left\| u\right\| _{\dot{H}_{r}^{s} }^{\max \{ 0,\;\sigma -\left\lceil s\right\rceil \} } +\left\| v\right\| _{\dot{H}_{r}^{s} }^{\max \{ 0,\;\sigma -\left\lceil s\right\rceil \} } \right).
\end{split}\end{eqnarray}
In view of (3.24) and (3.25), we have
\begin{equation} \label{GrindEQ__3_26_}
\left\| I_{1} \right\| _{p} \lesssim\left\| u-v\right\| _{a}^{\min \{ \sigma -\left\lceil s\right\rceil +1,\;1\} } \left(\left\| u\right\| _{\dot{H}_{r}^{s} }^{\max \{ \left\lceil s\right\rceil ,\;\sigma \} } +\left\| v\right\| _{\dot{H}_{r}^{s} }^{\max \{ \left\lceil s\right\rceil ,\;\sigma \} } \right).
\end{equation}

$\cdot$ If $q<\left|\alpha \right|=\left\lceil s\right\rceil $, then we have
\begin{eqnarray}\begin{split} \label{GrindEQ__3_27_}
\left\| f^{\left(q\right)} \left(u\right)-f^{\left(q\right)} \left(v\right)\right\| _{p_{5} }&=\left\| \left(u-v\right)\int _{0}^{1}f^{\left(q+1\right)} \left(v+t\left(u-v\right)\right)dt \right\| _{p_{5}}\\
&\lesssim \left\| \left(u-v\right)\int _{0}^{1}\left|v+t\left(u-v\right)\right|^{\sigma -q} dt \right\| _{p_{5}}\\
&\lesssim\left\| u-v\right\| _{a} \left\| \left|u\right|+\left|v\right|\right\| _{a}^{\sigma -q} \lesssim\left\| u-v\right\| _{\dot{H}_{r}^{s} } \left(\left\| u\right\| _{\dot{H}_{r}^{s} }^{\sigma -q} +\left\| v\right\| _{\dot{H}_{r}^{s} }^{\sigma -q} \right).
\end{split}\end{eqnarray}
In view of (3.24) and (3.27), we have
\begin{equation} \label{GrindEQ__3_28_}
\left\| I_{1} \right\| _{p} \lesssim\left(\left\| u\right\| _{\dot{H}_{r}^{s} }^{\sigma } +\left\| v\right\| _{\dot{H}_{r}^{s} }^{\sigma } \right)\left\| u-v\right\| _{\dot{H}_{r}^{s} } .
\end{equation}
Hence we have
\begin{eqnarray}\begin{split} \label{GrindEQ__3_29_}
\left\| I_{1} \right\| _{p} \lesssim&\left\| u-v\right\| _{a}^{\min \{ \sigma -\left\lceil s\right\rceil +1,\;1\} } \left(\left\| u\right\| _{\dot{H}_{r}^{s} }^{\max \{ \left\lceil s\right\rceil ,\;\sigma \} } +\left\| v\right\| _{\dot{H}_{r}^{s} }^{\max \{ \left\lceil s\right\rceil ,\;\sigma \} } \right)\\
&+\left(\left\| u\right\| _{\dot{H}_{r}^{s} }^{\sigma } +\left\| v\right\| _{\dot{H}_{r}^{s} }^{\sigma } \right)\left\| u-v\right\| _{\dot{H}_{r}^{s} } ,
\end{split}\end{eqnarray}
for any $1\le q\le s=\left\lceil s\right\rceil $.

\emph{Step 1.2.} Next, we estimate $\left\| II_{1} \right\| _{p}$, where $II_{1}$ is given in (3.21). Notice that
\begin{equation} \label{GrindEQ__3_30_}
\prod _{i=1}^{N}a_{i}  -\prod _{i=1}^{N}b_{i}  =\sum _{i=1}^{N}\prod _{j=1}^{i-1}a_{j}  \prod _{j=i+1}^{N}b_{j}  \left(a_{i} -b_{i} \right) ,
\end{equation}
where we assume that $\prod _{j=1}^{0}a_{i}  =\prod _{j=N+1}^{N}b_{i}  =0$. Hence we have
\begin{eqnarray}\begin{split} \label{GrindEQ__3_31_}
\left\| II_{1} \right\| _{p} &=\left\| f^{\left(q\right)} \left(v\right)\left(\prod _{i=1}^{q}D^{\alpha _{i} } u -\prod _{i=1}^{q}D^{\alpha _{i} } v \right)\right\| _{p}\\
&\le \sum _{i=1}^{q}\left\| f^{\left(q\right)} \left(v\right)\left(\prod _{j=1}^{i-1}D^{\alpha _{j} } u \prod _{j=i+1}^{q}D^{\alpha _{j} } v \left(D^{\alpha _{i} } u-D^{\alpha _{i} } v\right)\right)\right\| _{p}  .
\end{split}\end{eqnarray}
It follows from (3.15), (3.22), (3.23) and H\"{o}lder inequality that
\begin{eqnarray}\begin{split} \label{GrindEQ__3_32_}
&\left\| f^{\left(q\right)} \left(v\right)\left(\prod _{j=1}^{i-1}D^{\alpha _{j} } u \prod _{j=i+1}^{q}D^{\alpha _{j} } v \left(D^{\alpha _{i} } u-D^{\alpha _{i} } v\right)\right)\right\| _{p}\\
&~~~~~~~~~~~\lesssim \left\| u\right\| _{a}^{\sigma +1-q} \left\| u-v\right\| _{\dot{H}_{a_{i} }^{\left|\alpha _{i} \right|} } \prod _{j=1}^{i-1}\left\| u\right\| _{\dot{H}_{a_{j} }^{\left|\alpha _{j} \right|} }  \prod _{j=i+1}^{q}\left\| v\right\| _{\dot{H}_{a_{j} }^{\left|\alpha _{j} \right|} }\\
&~~~~~~~~~~~\lesssim\left(\left\| u\right\| _{\dot{H}_{r}^{s} }^{\sigma } +\left\| v\right\| _{\dot{H}_{r}^{s} }^{\sigma } \right)\left\| u-v\right\| _{\dot{H}_{r}^{s} } .
\end{split}\end{eqnarray}
In view of (3.31) and (3.32), we have
\begin{equation} \label{GrindEQ__3_33_}
\left\| II_{1} \right\| _{p} \lesssim\left(\left\| u\right\| _{\dot{H}_{r}^{s} }^{\sigma } +\left\| v\right\| _{\dot{H}_{r}^{s} }^{\sigma } \right)\left\| u-v\right\| _{\dot{H}_{r}^{s} } .
\end{equation}
It follows from (3.20), (3.29) and (3.33) that
\begin{eqnarray}\begin{split}\nonumber
\left\| D^{\alpha } f(u)-D^{\alpha } f(v)\right\| _{\dot{H}_{p}^{v} } &\lesssim\left\| u-v\right\| _{a}^{\min \{ \sigma -\left\lceil s\right\rceil +1,\;1\} } \left(\left\| u\right\| _{\dot{H}_{r}^{s} }^{\max \{ \left\lceil s\right\rceil ,\;\sigma \} } +\left\| v\right\| _{\dot{H}_{r}^{s} }^{\max \{ \left\lceil s\right\rceil ,\;\sigma \} } \right)\\
&~~+\left(\left\| u\right\| _{\dot{H}_{r}^{s} }^{\sigma } +\left\| v\right\| _{\dot{H}_{r}^{s} }^{\sigma } \right)\left\| u-v\right\| _{\dot{H}_{r}^{s} } ,
\end{split}\end{eqnarray}
this completes the proof of (3.17) in the case $s\in \N$.

\textbf{Case 2.} $s\notin \N$.

\emph{Step 2.1.} First, we estimate $\left\| I_{1} \right\| _{\dot{H}_{p}^{v} } $, where $I_{1}$ is given in (3.21).

It follows from (3.22) and Lemma 2.1 (fractional product rule) that
\begin{eqnarray}\begin{split} \label{GrindEQ__3_34_}
\left\| I_{1} \right\| _{\dot{H}_{p}^{v} } &\lesssim\left\| f^{\left(q\right)} \left(u\right)-f^{\left(q\right)} \left(v\right)\right\| _{\dot{H}_{p_{4} }^{v} } \left\| \prod _{i=1}^{q}D^{\alpha _{i} } u \right\| _{r_{4} } +\left\| f^{\left(q\right)} \left(u\right)-f^{\left(q\right)} \left(v\right)\right\| _{p_{5} } \left\| \prod _{i=1}^{q}D^{\alpha _{i} } u \right\| _{\dot{H}_{r_{5} }^{v} } \\
&\equiv I_{2} +II_{2} ,
\end{split}\end{eqnarray}
where
\begin{equation} \label{GrindEQ__3_35_}
\frac{1}{p_{4} } =\frac{\sigma +1-q}{a} +\frac{v}{n},~\frac{1}{r_{4} } =\sum _{i=1}^{q}\frac{1}{a_{i}},~\frac{1}{p_{5} } =\frac{\sigma +1-q}{a},~\frac{1}{r_{5}} =\sum _{i=1}^{q}\frac{1}{a_{i}}+\frac{v}{n} .
\end{equation}
First, we estimate
\[I_{2} =\left\| f^{\left(q\right)} \left(u\right)-f^{\left(q\right)} \left(v\right)\right\| _{\dot{H}_{p_{4} }^{v} } \left\| \prod _{i=1}^{q}D^{\alpha _{i} } u \right\| _{r_{4} } .\]
Using the embedding $\dot{H}_{r}^{s} \subset \dot{H}_{a_{i} }^{\left|\alpha _{i} \right|} $, we can easily see that
\begin{equation} \label{GrindEQ__3_36_}
\left\| \prod _{i=1}^{q}D^{\alpha _{i} } u \right\| _{r_{4} } \lesssim \left\| u\right\| _{\dot{H}_{r}^{s} }^{q} .
\end{equation}
Putting $\frac{1}{p_{6} } =\frac{\sigma +1-q}{a} +\frac{1}{n} $, we can see that
\[0<\frac{1}{p_{6} } \le \sigma \left(\frac{1}{r} -\frac{s}{n} \right)+\frac{1}{n} \le \sigma \left(\frac{1}{r} -\frac{s}{n} \right)+\frac{s}{n} <\sigma \left(\frac{1}{r} -\frac{s}{n} \right)+\frac{1}{r} =\frac{1}{p_{1} } <1,\]
which implies that $1<p_{6} <\infty $ and $\dot{H}_{p_{6} }^{1} \subset \dot{H}_{p_{4} }^{v} $. Thus we have
\begin{eqnarray}\begin{split} \label{GrindEQ__3_37_}
&\left\| f^{\left(q\right)} \left(u\right)-f^{\left(q\right)} \left(v\right)\right\| _{\dot{H}_{p_{4} }^{v} } \lesssim\left\| f^{\left(q\right)} \left(u\right)-f^{\left(q\right)} \left(v\right)\right\| _{\dot{H}_{p_{6} }^{1} }\\
&~~~~~~=\sum _{i=1}^{n}\left\| \partial _{x_{i} } \left(f^{\left(q\right)} \left(u\right)-f^{\left(q\right)} \left(v\right)\right)\right\| _{p_{6} }  \\
&~~~~~~=\sum _{i=1}^{n}\left\| f^{\left(q+1\right)} \left(u\right)\partial _{x_{i} } u-f^{\left(q+1\right)} \left(v\right)\partial _{x_{i} } v\right\| _{p_{6} }   \\
&~~~~~~\le \sum _{i=1}^{n}\left\| \left(f^{\left(q+1\right)} \left(u\right)-f^{\left(q+1\right)} \left(v\right)\right)\partial _{x_{i} } u\right\| _{p_{6} } +\left\| f^{\left(q+1\right)} \left(v\right)\left(\partial _{x_{i} } u-\partial _{x_{i} } v\right)\right\| _{p_{6} }   \\
&~~~~~~\lesssim\sum _{i=1}^{n}\left\| f^{\left(q+1\right)} \left(u\right)-f^{\left(q+1\right)} \left(v\right)\right\| _{p_{7} } \left\| \partial _{x_{i} } u\right\| _{r_{7} } +\left\| f^{\left(q+1\right)} \left(v\right)\right\| _{p_{7} } \left(\left\| \partial _{x_{i} } u-\partial _{x_{i} } v\right\| _{r_{7} } \right),
\end{split}\end{eqnarray}
where
\begin{equation} \label{GrindEQ__3_38_}
\frac{1}{p_{7} } =\frac{\sigma -q}{a} ,~\frac{1}{r_{7} } =\frac{1}{r} -\frac{s-1}{n} .
\end{equation}
Using the embedding $\dot{H}_{r_{7} }^{1} \subset \dot{H}_{r}^{s} $, we have
\begin{equation} \label{GrindEQ__3_39_}
\left\| \partial _{x_{i} } u\right\| _{r_{7} } \lesssim\left\| u\right\| _{\dot{H}_{r_{7} }^{1} } \lesssim\left\| u\right\| _{\dot{H}_{r}^{s} } ,~ \left\| \partial _{x_{i} } u-\partial _{x_{i} } v\right\| _{r_{7} } \lesssim\left\| u-v\right\| _{\dot{H}_{r_{7} }^{1} } \lesssim\left\| u-v\right\| _{\dot{H}_{r}^{s} } .
\end{equation}
We can also easily see that
\begin{equation} \label{GrindEQ__3_40_}
\left\| f^{\left(q+1\right)} \left(v\right)\right\| _{p_{7} } \lesssim \left\| v\right\| _{\dot{H}_{r}^{s} }^{\sigma -q} .
\end{equation}

$\cdot$ If $q=\left[s\right]=\left\lceil s\right\rceil -1$, it follows from (3.16) that
\begin{equation} \label{GrindEQ__3_41_}
\left\| f^{\left(\left\lceil s\right\rceil \right)} \left(u\right)-f^{\left(\left\lceil s\right\rceil \right)} \left(v\right)\right\| _{p_{7} } \lesssim\left\| u-v\right\| _{a}^{\min \{ \sigma -\left\lceil s\right\rceil +1,\;1\} } \left(\left\| u\right\| _{\dot{H}_{r}^{s} }^{\max \{ 0,\;\sigma -\left\lceil s\right\rceil \} } +\left\| v\right\| _{\dot{H}_{r}^{s} }^{\max \{ 0,\;\sigma -\left\lceil s\right\rceil \} } \right).
\end{equation}
In view of (3.37), (3.39), (3.40) and (3.41), we have
\begin{eqnarray}\begin{split} \label{GrindEQ__3_42_}
\left\| f^{\left(\left[s\right]\right)} \left(u\right)-f^{\left(\left[s\right]\right)} \left(v\right)\right\| _{\dot{H}_{p_{4} }^{v} } &\lesssim\left\| u-v\right\| _{a}^{\min \{ \sigma -\left\lceil s\right\rceil +1,\;1\} } \left(\left\| u\right\| _{\dot{H}_{r}^{s} }^{\max \{ 1,\;\sigma -\left[s\right]\} } +\left\| v\right\| _{\dot{H}_{r}^{s} }^{\max \{ 1,\;\sigma -\left[s\right]\} } \right) \\
&~~+\left(\left\| u\right\| _{\dot{H}_{r}^{s} }^{\sigma -\left[s\right]} +\left\| v\right\| _{\dot{H}_{r}^{s} }^{\sigma -\left[s\right]} \right)\left\| u-v\right\| _{\dot{H}_{r}^{s} } .
\end{split}\end{eqnarray}
It follows from (3.36) and (3.42) that
\begin{eqnarray}\begin{split} \label{GrindEQ__3_43_}
I_{2} &=\left\| f^{\left(\left[s\right]\right)} \left(u\right)-f^{\left(\left[s\right]\right)} \left(v\right)\right\| _{\dot{H}_{p_{4} }^{v} } \left\| \prod _{i=1}^{\left[s\right]}D^{\alpha _{i} } u \right\| _{r_{4} }  \\
&\lesssim\left\| u-v\right\| _{a}^{\min \{ \sigma -\left\lceil s\right\rceil +1,\;1\} } \left(\left\| u\right\| _{\dot{H}_{r}^{s} }^{\max \{ \left\lceil s\right\rceil ,\;\sigma \} } +\left\| v\right\| _{\dot{H}_{r}^{s} }^{\max \{ \left\lceil s\right\rceil ,\;\sigma \} } \right)+\left(\left\| u\right\| _{\dot{H}_{r}^{s} }^{\sigma }+\left\| v\right\| _{\dot{H}_{r}^{s} }^{\sigma } \right)\left\| u-v\right\| _{\dot{H}_{r}^{s} }
\end{split}\end{eqnarray}

$\cdot$ If $q<\left[s\right]=\left\lceil s\right\rceil -1$, it follows from (3.15) that
\begin{eqnarray}\begin{split} \label{GrindEQ__3_44_}
\left\| f^{\left(q+1\right)} \left(u\right)-f^{\left(q+1\right)} \left(v\right)\right\| _{p_{7} } &=\left\| \left(u-v\right)\int _{0}^{1}f^{\left(q+2\right)} \left(v+t\left(u-v\right)\right)dt \right\| _{p_{7} }\\
&\lesssim \left\| \left(u-v\right)\int _{0}^{1}\left|v+t\left(u-v\right)\right|^{\sigma -q-1} dt \right\| _{p_{7} }\\
&\lesssim\left\| u-v\right\| _{a} \left\| \left|u\right|+\left|v\right|\right\| _{a}^{\sigma -q-1}\\
&\lesssim\left\| u-v\right\| _{\dot{H}_{r}^{s} } \left(\left\| u\right\| _{\dot{H}_{r}^{s} }^{\sigma -q-1} +\left\| v\right\| _{\dot{H}_{r}^{s} }^{\sigma -q-1} \right).
\end{split}\end{eqnarray}
In view of (3.37), (3.39), (3.40) and (3.44), we have
\begin{equation} \label{GrindEQ__3_45_}
\left\| f^{\left(q\right)} \left(u\right)-f^{\left(q\right)} \left(v\right)\right\| _{\dot{H}_{p_{4} }^{v} } \lesssim\left(\left\| u\right\| _{\dot{H}_{r}^{s} }^{\sigma -q} +\left\| v\right\| _{\dot{H}_{r}^{s} }^{\sigma -q} \right)\left\| u-v\right\| _{\dot{H}_{r}^{s} } .
\end{equation}
(3.36) and (3.45) yield that
\begin{equation} \label{GrindEQ__3_46_}
I_{2} \lesssim \left(\left\| u\right\| _{\dot{H}_{r}^{s} }^{\sigma } +\left\| v\right\| _{\dot{H}_{r}^{s} }^{\sigma } \right)\left\| u-v\right\| _{\dot{H}_{r}^{s} } .
\end{equation}
Thus, for any $1\le q\le \left[s\right]$, we have
\begin{eqnarray}\begin{split} \label{GrindEQ__3_47_}
I_{2} \lesssim &\left\| u-v\right\| _{a}^{\min \{ \sigma -\left\lceil s\right\rceil +1,\;1\} } \left(\left\| u\right\| _{\dot{H}_{r}^{s} }^{\max \{ \left\lceil s\right\rceil ,\;\sigma \} } +\left\| v\right\| _{\dot{H}_{r}^{s} }^{\max \{ \left\lceil s\right\rceil ,\;\sigma \} } \right)+\left(\left\| u\right\| _{\dot{H}_{r}^{s} }^{\sigma }+\left\| v\right\| _{\dot{H}_{r}^{s} }^{\sigma } \right)\left\| u-v\right\| _{\dot{H}_{r}^{s}}.
\end{split}\end{eqnarray}
Next, we estimate
\[II_{2} =\left\| f^{\left(q\right)} \left(u\right)-f^{\left(q\right)} \left(v\right)\right\| _{p_{5} } \left\| \prod _{i=1}^{q}D^{\alpha _{i} } u \right\| _{\dot{H}_{r_{5} }^{v} } ,\]
where $p_{5}$ and $r_{5}$ are given in (3.35).
If $q=1$, we can see that $\left|\alpha _{1} \right|=\left[s\right]$ and $r_{5} =r$. Hence, we have
\[\left\| D^{\alpha _{1} } u\right\| _{\dot{H}_{r_{5} }^{v} } \lesssim\left\| u\right\| _{\dot{H}_{r}^{s} } .\]
We consider the case $q>1$. For $1\le k\le q$, putting $\frac{1}{\tilde{a}_{k} } :=\frac{1}{a_{k} } +\frac{v}{n} $, it follows from (3.23) and (3.35) that
\begin{equation} \label{GrindEQ__3_48_}
\frac{1}{\tilde{a}_{k} } =\frac{1}{r} -\frac{s-\left|\alpha _{k} \right|-v}{n} , ~\frac{1}{r_{5} } =\sum _{i\in I_{k} }\frac{1}{a_{i} }  +\frac{1}{\tilde{a}_{k} } ,
\end{equation}
where $I_{k} =\left\{i\in \N:\;1\le i\le q,\;i\ne k\right\}$. We can see that $\tilde{a}_{k} >r>1$ and $\dot{H}_{r}^{s} \subset \dot{H}_{\tilde{a}_{k} }^{\left|\alpha _{k} \right|+v} $, since $s>\left|\alpha _{k} \right|+v$. Hence, using Corollary 2.2 and (3.48), we have
\begin{eqnarray}\begin{split} \label{GrindEQ__3_49_}
\left\| \prod _{i=1}^{q}D^{\alpha _{i} } u \right\| _{\dot{H}_{r_{5} }^{v} } &\lesssim \sum _{k=1}^{q}\left(\left\| D^{\alpha _{k} } u\right\| _{\dot{H}_{\tilde{a}_{k} }^{v} } \prod _{i\in I_{k} }\left\| D^{\alpha _{i} } u_{i} \right\| _{a_{i} }  \right)  \\
&\lesssim\sum _{k=1}^{q}\left(\left\| u\right\| _{\dot{H}_{\tilde{a}_{k} }^{\left|\alpha _{k} \right|+v} } \prod _{i\in I_{k} }\left\| u\right\| _{\dot{H}_{a_{i} }^{\left|\alpha _{i} \right|} } \right) \lesssim\left\| u\right\| _{\dot{H}_{r}^{s} }^{q} .
\end{split}\end{eqnarray}
Thus, for any $1\le q\le \left[s\right]$, we have
\begin{equation} \label{GrindEQ__3_50_}
\left\| \prod _{i=1}^{q}D^{\alpha _{i} } u \right\| _{\dot{H}_{r_{5} }^{v} } \lesssim\left\| u\right\| _{\dot{H}_{r}^{s} }^{q} .
\end{equation}
Since $q\le \left[s\right]<\left\lceil s\right\rceil $, it follows from (3.15) that
\begin{eqnarray}\begin{split} \label{GrindEQ__3_51_}
\left\| f^{\left(q\right)} \left(u\right)-f^{\left(q\right)} \left(v\right)\right\| _{p_{5} } &=\left\| \left(u-v\right)\int _{0}^{1}f^{\left(q+1\right)} \left(v+t\left(u-v\right)\right)dt \right\| _{p_{5} }\\
&\lesssim \left\| \left(u-v\right)\int _{0}^{1}\left|v+t\left(u-v\right)\right|^{\sigma -q} dt \right\| _{p_{5} }  \\
&\lesssim\left\| u-v\right\| _{a} \left\| \left|u\right|+\left|v\right|\right\| _{a}^{\sigma -q} \\
&\lesssim\left\| u-v\right\| _{\dot{H}_{r}^{s} } \left(\left\| u\right\| _{\dot{H}_{r}^{s} }^{\sigma -q} +\left\| v\right\| _{\dot{H}_{r}^{s} }^{\sigma -q} \right).
\end{split}\end{eqnarray}
In view of (3.50) and (3.51), we have
\begin{equation} \label{GrindEQ__3_52_}
II_{2} =\left\| f^{\left(q\right)} \left(u\right)-f^{\left(q\right)} \left(v\right)\right\| _{p_{5} } \left\| \prod _{i=1}^{q}D^{\alpha _{i} } u \right\| _{\dot{H}_{r_{5} }^{v} } \lesssim\left(\left\| u\right\| _{\dot{H}_{r}^{s} }^{\sigma } +\left\| v\right\| _{\dot{H}_{r}^{s} }^{\sigma } \right)\left\| u-v\right\| _{\dot{H}_{r}^{s} } .
\end{equation}
In view of (3.34), (3.47) and (3.52), we have
\begin{eqnarray}\begin{split} \label{GrindEQ__3_53_}
&\left\| I_{1} \right\| _{\dot{H}_{p}^{v} }\lesssim I_{2}+II_{2}\\
&~~~~~\lesssim\left\| u-v\right\| _{a}^{\min \{ \sigma -\left\lceil s\right\rceil +1,\;1\} } \left(\left\| u\right\| _{\dot{H}_{r}^{s} }^{\max \{ \left\lceil s\right\rceil ,\;\sigma \} } +\left\| v\right\| _{\dot{H}_{r}^{s} }^{\max \{ \left\lceil s\right\rceil ,\;\sigma \} } \right)+\left(\left\| u\right\| _{\dot{H}_{r}^{s} }^{\sigma } +\left\| v\right\| _{\dot{H}_{r}^{s} }^{\sigma } \right)\left\| u-v\right\| _{\dot{H}_{r}^{s} } ,
\end{split}\end{eqnarray}
this completes the estimate of $\left\| I_{1} \right\| _{\dot{H}_{p}^{v} }$.

\emph{Step 2.2.} Next, we estimate $\left\| II_{1} \right\| _{\dot{H}_{p}^{v} } $, where $II_{1}$ is given in (3.21).
It follows from (3.22) and Lemma 2.1 (fractional product rule) that
\begin{eqnarray}\begin{split} \label{GrindEQ__3_54_}
\left\| II_{1} \right\| _{\dot{H}_{p}^{v} } \lesssim& \left\| f^{\left(q\right)} \left(v\right)\right\| _{\dot{H}_{p_{4} }^{v} } \left\| \prod _{i=1}^{q}D^{\alpha _{i} } u -\prod _{i=1}^{q}D^{\alpha _{i} } v \right\| _{r_{4} }+\left\| f^{\left(q\right)} \left(v\right)\right\| _{p_{5} } \left\| \prod _{i=1}^{q}D^{\alpha _{i} } u -\prod _{i=1}^{q}D^{\alpha _{i} } v \right\| _{\dot{H}_{r_{5} }^{v} } ,
\end{split}\end{eqnarray}
where $p_{4}$, $p_{5}$, $r_{4}$ and $r_{5}$ are given in (3.35). Putting $\frac{1}{r_{8} } =\frac{1}{r} -\frac{\left[s\right]}{n} $ and using lemma 2.4, we have the embedding $\dot{H}_{r}^{s} \subset \dot{H}_{r_{8} }^{v} $. Noticing that $\frac{1}{p_{4} } =\frac{1}{p_{7} } +\frac{1}{r_{8}}$, it follows from (3.15), lemma 2.3 (fractional chain rule) and H\"{o}lder inequality that
\begin{equation} \label{GrindEQ__3_55_}
\left\| f^{\left(q\right)} \left(v\right)\right\| _{\dot{H}_{p_{4} }^{v} } \lesssim\left\| f^{\left(q+1\right)} \left(v\right)\right\| _{p_{7} } \left\| v\right\| _{\dot{H}_{r_{8} }^{v} } \lesssim\left\| v\right\| _{a}^{\sigma -q} \left\| v\right\| _{\dot{H}_{r_{8} }^{v} } \lesssim\left\| v\right\| _{\dot{H}_{r}^{s} }^{\sigma +1-q} ,
\end{equation}
where $a$ and $p_{7}$ are given in (3.23) and (3.38) respectively. By (3.15), we also have
\begin{equation} \label{GrindEQ__3_56_}
\left\| f^{\left(q\right)} \left(v\right)\right\| _{p_{5} } \lesssim\left\| \left|v\right|^{\sigma +1-q} \right\| _{p_{5} } \lesssim\left\| v\right\| _{a}^{\sigma +1-q} \lesssim\left\| v\right\| _{\dot{H}_{r}^{s} }^{\sigma +1-q} .
\end{equation}
It follows from (3.30), (3.35) and H\"{o}lder inequality that
\begin{eqnarray}\begin{split} \label{GrindEQ__3_57_}
\left\| \prod _{i=1}^{q}D^{\alpha _{i} } u -\prod _{i=1}^{q}D^{\alpha _{i} } v \right\| _{r_{4} } &\le \sum _{i=1}^{q}\left\| \prod _{j=1}^{i-1}D^{\alpha _{j} } u \prod _{j=i+1}^{q}D^{\alpha _{j} } v \left(D^{\alpha _{i} } u-D^{\alpha _{i} } v\right)\right\| _{p_{4}}\\
&\lesssim \sum _{i=1}^{q}\left\| u-v\right\| _{\dot{H}_{a_{i} }^{\left|\alpha _{i} \right|} } \prod _{j=1}^{i-1}\left\| u\right\| _{\dot{H}_{a_{j} }^{\left|\alpha _{j} \right|} }  \prod _{j=i+1}^{q}\left\| v\right\| _{\dot{H}_{a_{j} }^{\left|\alpha _{j}\right|}}\\
&\lesssim\left(\left\| u\right\| _{\dot{H}_{r}^{s} }^{q-1} +\left\| v\right\| _{\dot{H}_{r}^{s} }^{q-1} \right)\left\| u-v\right\| _{\dot{H}_{r}^{s} } .
\end{split}\end{eqnarray}
We also have
\begin{equation} \label{GrindEQ__3_58_}
\left\| \prod _{i=1}^{q}D^{\alpha _{i} } u -\prod _{i=1}^{q}D^{\alpha _{i} } v \right\| _{\dot{H}_{r_{5} }^{v} } \le \sum _{i=1}^{q}\left\| \prod _{j=1}^{i-1}D^{\alpha _{j} } u \prod _{j=i+1}^{q}D^{\alpha _{j} } v \left(D^{\alpha _{i} } u-D^{\alpha _{i} } v\right)\right\| _{\dot{H}_{r_{5} }^{v} }  .
\end{equation}
Using Corollary 2.2, (3.48) and repeating the same argument as in the estimate of (3.49), we have
\begin{equation} \label{GrindEQ__3_59_}
\left\| \prod _{j=1}^{i-1}D^{\alpha _{j} } u \prod _{j=i+1}^{q}D^{\alpha _{j} } v \left(D^{\alpha _{i} } u-D^{\alpha _{i} } v\right)\right\| _{\dot{H}_{r_{5} }^{v} } \lesssim\left(\left\| u\right\| _{\dot{H}_{r}^{s} }^{q-1} +\left\| v\right\| _{\dot{H}_{r}^{s} }^{q-1} \right)\left\| u-v\right\| _{\dot{H}_{r}^{s} } ,
\end{equation}
whose proof will be omitted. In view of (3.58) and (3.59), we have
\begin{equation} \label{GrindEQ__3_60_}
\left\| \prod _{i=1}^{q}D^{\alpha _{i} } u -\prod _{i=1}^{q}D^{\alpha _{i} } v \right\| _{\dot{H}_{r_{5} }^{v} } \lesssim\left(\left\| u\right\| _{\dot{H}_{r}^{s} }^{q-1} +\left\| v\right\| _{\dot{H}_{r}^{s} }^{q-1} \right)\left\| u-v\right\| _{\dot{H}_{r}^{s} } .
\end{equation}
Using (3.54)--(3.57) and (3.60), we have
\begin{equation} \label{GrindEQ__3_61_}
\left\| II_{1} \right\| _{\dot{H}_{p}^{v} } \lesssim\left(\left\| u\right\| _{\dot{H}_{r}^{s} }^{\sigma } +\left\| v\right\| _{\dot{H}_{r}^{s} }^{\sigma } \right)\left\| u-v\right\| _{\dot{H}_{r}^{s} } .
\end{equation}
In view of (3.20), (3.53) and (3.61), we have
\begin{eqnarray}\begin{split} \label{GrindEQ__3_62_}
\left\| D^{\alpha } f(u)-D^{\alpha } f(v)\right\| _{\dot{H}_{p}^{v} } \lesssim&\left\| u-v\right\| _{a}^{\min \{ \sigma -\left\lceil s\right\rceil +1,\;1\} } \left(\left\| u\right\| _{\dot{H}_{r}^{s} }^{\max \{ \left\lceil s\right\rceil ,\;\sigma \} } +\left\| v\right\| _{\dot{H}_{r}^{s} }^{\max \{ \left\lceil s\right\rceil ,\;\sigma \}} \right)\\
&+\left(\left\| u\right\| _{\dot{H}_{r}^{s} }^{\sigma } +\left\| v\right\| _{\dot{H}_{r}^{s} }^{\sigma } \right)\left\| u-v\right\|_{\dot{H}_{r}^{s} } ,
\end{split}\end{eqnarray}
this completes the proof of (3.17). If $\sigma \ge \left\lceil s\right\rceil $, then (3.18) follows directly from (3.17) and the embedding $\dot{H}_{r}^{s} \subset L^{a}$.
\end{proof}

\begin{lem}\label{lem 3.3.}
Let $s>0$ and $f\left(z\right)$ be a polynomial in $z$ and $\bar{z}$ satisfying $1<\deg \left(f\right)=1+\sigma $. Suppose also that \textnormal{(3.4)} holds. Then we have
\begin{equation} \label{GrindEQ__3_63_}
\left\| f(u)-f(v)\right\| _{\dot{H}_{p}^{s} } \lesssim\left(\left\| u\right\| _{\dot{H}_{r}^{s} }^{\sigma } +\left\| v\right\| _{\dot{H}_{r}^{s} }^{\sigma } \right)\left\| u-v\right\| _{\dot{H}_{r}^{s} } .
\end{equation}
\end{lem}
\begin{proof}
If $\sigma \ge \left\lceil s\right\rceil $, (3.63) follows directly from Lemma 3.1 and Lemma 3.2. If $\left\lceil s\right\rceil -1\ge \sigma $, we have
\begin{equation} \label{GrindEQ__3_64_}
\left|f^{\left(k\right)} \left(u\right)\right|\lesssim\left|u\right|^{\sigma +1-k} ,
\end{equation}
for any $0\le k\le \sigma +1$ and
\begin{equation} \label{GrindEQ__3_65_}
\left|f^{\left(k\right)} \left(u\right)\right|=0,
\end{equation}
for $k>\sigma +1$. Using (3.64) and (3.65), we have
\begin{equation} \label{GrindEQ__3_66_}
\left|f^{(q)} \left(u\right)-f^{(q)} \left(v\right)\right|\;\lesssim\left|u-v\right|\left(\left|u\right|^{\sigma -q} +\left|v\right|^{\sigma -q} \right),
\end{equation}
for any $q\le \sigma $ and
\begin{equation} \label{GrindEQ__3_67_}
\left|f^{(q)} \left(u\right)-f^{(q)} \left(v\right)\right|=0,
\end{equation}
for any $q>\sigma $. Using (3.64)--(3.67) and the same argument as in the proof of Lemma 3.2, we can easily prove (3.63) in the case $\left\lceil s\right\rceil \ge \sigma +1$ and we omit the details.
\end{proof}

\begin{lem}\label{lem 3.4.}
Let $p>1$, $s>0$ and $\sigma >0$. Assume that $f\in C^{1} \left(\C\to \C\right)$ satisfies
\begin{equation} \label{GrindEQ__3_68_}
\left|f'\left(u\right)\right|\lesssim\left|u\right|^{\sigma } ,
\end{equation}
for any $u\in \C$. Suppose also that
\[\frac{1}{p} =\left(\sigma +1\right)\left(\frac{1}{r} -\frac{s}{n} \right),~\frac{1}{r} -\frac{s}{n} >0.\]
Then we have
\[\left\| f(u)-f(v)\right\| _{p} \lesssim\left(\left\| u\right\| _{\dot{H}_{r}^{s} }^{\sigma } +\left\| v\right\| _{\dot{H}_{r}^{s} }^{\sigma } \right)\left\| u-v\right\| _{\dot{H}_{r}^{s} } .\]
\end{lem}
\begin{proof}
Since
\[f(u)-f(v)=\int _{0}^{1}\left[f_{z} \left(v+t\left(u-v\right)\right)\left(u-v\right)+f_{\bar{z}} \left(v+t\left(u-v\right)\right)\left(\bar{u}-\bar{v}\right)\right]dt ,\]
it follows from (3.68) that
\[\left|f(u)-f(v)\right|\lesssim\left(\left|u\right|^{\sigma } +\left|v\right|^{\sigma } \right)\left|u-v\right|.\]
Thus we have
\begin{eqnarray}\begin{split}\nonumber
\left\| f(u)-f(v)\right\| _{p}&\lesssim\left\| \left(\left|u\right|^{\sigma } +\left|v\right|^{\sigma } \right)\left(u-v\right)\right\| _{p} \le \left(\left\| u\right\| _{a}^{\sigma } +\left\| v\right\| _{a}^{\sigma } \right)\left\| u-v\right\| _{a}\\
&\lesssim\left(\left\| u\right\| _{\dot{H}_{r}^{s} }^{\sigma } +\left\| v\right\| _{\dot{H}_{r}^{s} }^{\sigma } \right)\left\| u-v\right\| _{\dot{H}_{r}^{s} },
\end{split}\end{eqnarray}
where $\frac{1}{a} =\frac{1}{r} -\frac{s}{n} $ and the last inequality follows from the embedding $\dot{H}_{r}^{s} \subset L^{a} $.
\end{proof}

\begin{rem}\label{rem 3.5.}
\textnormal{Let $B=B\left(0,\;1\right)=\left\{x\in \R^{n} ;\;|x|\le 1\right\}$, $b>0$ and $s\ge0$. If $\frac{n}{\gamma }>b+s$, then $\chi_{B}|x|^{-b} \in \dot{H}_{\gamma }^{s}(\R^{n})$. In fact, an easy computation shows that $\chi_{B}|x|^{-b} \in L^{\gamma}(\R^{n})$, if $\frac{n}{\gamma }>b$ (see Remark 2.8 of \cite{G17}). Putting $\frac{n}{\gamma_{1}}:=\left\lceil s\right\rceil-s+\frac{n}{\gamma}$, we have $\frac{n}{\gamma_{1}}>b+\left\lceil s\right\rceil$. Thus, it follows from Lemma 2.4 and Lemma 2.7 that
\begin{eqnarray}\begin{split}\nonumber
\left\|\chi_{B}|x|^{-b}\right\|_{\dot{H}_{\gamma }^{s}(\R^{n})}&\lesssim\left\|\chi_{B}|x|^{-b}\right\|_{\dot{H}_{\gamma_{1} }^{\left\lceil s\right\rceil}(\R^{n})}\lesssim \sum _{\left|\alpha \right|=\left\lceil s\right\rceil}\left\| D^{\alpha }(\chi_{B}|x|^{-b})\right\| _{L^{\gamma_{1}}(\R^{n}) }\\
&\lesssim \left\| \chi_{B}|x|^{-b-\left\lceil s\right\rceil}\right\| _{L^{\gamma_{1}}(\R^{n})}<\infty
\end{split}\end{eqnarray}
Similarly, if $\frac{n}{\gamma }<b+s$, we have $\chi_{B^{C}}|x|^{-b}\in \dot{H}_{\gamma }^{s}(\R^{n})$.}
\end{rem}

Using Lemma 3.1--Lemma 3.4 and Remark 3.5, we have the following important estimates of the term $|x|^{-b} f(u)-|x|^{-b} f(v)$. We divide the study in two cases: $n\ge 3$ and $n=1,\, 2$. Here we use the similar argument as in the proof of Lemma 3.5 and Lemma 3.6 of \cite{AK21}. For the reader's convenience and for simplicity, we only give the sketch proof in the case $n\ge 3$.

\begin{lem}\label{lem 3.6.}
Let $n\ge 3$, $0<s<\frac{n}{2} $, $0<b<\min \left\{2,\; 1+\frac{n-2s}{2} \right\}$ and $0<\sigma <\frac{4-2b}{n-2s} $. Assume that $f$ is of class $C\left(\sigma ,s,b\right)$.

\textnormal{(1)} If $f\left(z\right)$ is a polynomial in $z$ and $\bar{z}$, or if not we assume further that $\sigma \ge \left\lceil s\right\rceil $, then we have
\begin{eqnarray}\begin{split} \label{GrindEQ__3_69_}
&\left\| |x|^{-b} f(u)-|x|^{-b} f(v)\right\| _{S'\left(I,\;\dot{H}^{s} \right)}\\
&~~~~~~~~~~~~~~~\lesssim\left(T^{\theta _{1} } +T^{\theta _{2} } \right)\left(\left\| u\right\| _{S\left(I,\;\dot{H}^{s} \right)}^{\sigma } +\left\| v\right\| _{S\left(I,\;\dot{H}^{s} \right)}^{\sigma } \right)\left\| u-v\right\| _{S\left(I,\;\dot{H}^{s} \right)} ,
\end{split}\end{eqnarray}
where $I=\left[-T,\, T\right]$ and $\theta _{1} ,\, \theta _{2} >0$.

\textnormal{(2)} If $f$ is not a polynomial and $\left\lceil s\right\rceil >\sigma > \left\lceil s\right\rceil -1$, then we have
\begin{eqnarray}\begin{split} \label{GrindEQ__3_70_}
&\left\| |x|^{-b} f(u)-|x|^{-b} f(v)\right\| _{S'\left(I,\;\dot{H}^{s} \right)}\\
&~~~~~~~~~~~~~~~~~~~~\lesssim\left(T^{\theta _{1} } +T^{\theta _{2} } \right)\left(\left\| u\right\| _{S\left(I,\;\dot{H}^{s} \right)}^{\sigma } +\left\| v\right\| _{S\left(I,\;\dot{H}^{s} \right)}^{\sigma } \right)\left\| u-v\right\| _{S\left(I,\;\dot{H}^{s} \right)}\\
&~~~~~~~~~~~~~~~~~~~~~~+\left(\left\| u\right\| _{L^{\gamma \left(r_{0} \right)} \left(I,\;\dot{H}_{r_{0} }^{s} \right)}^{\left\lceil s\right\rceil } +\left\| v\right\| _{L^{\gamma \left(r_{0} \right)} \left(I,\;\dot{H}_{r_{0} }^{s} \right)}^{\left\lceil s\right\rceil } \right)\left\| u-v\right\| _{L^{\bar{\gamma }\left(r_{0} \right)} \left(I,\;L^{a_{0} } \right)}^{\sigma +1-\left\lceil s\right\rceil }\\
&~~~~~~~~~~~~~~~~~~~~~~+\left(\left\| u\right\| _{L^{\gamma \left(r_{1} \right)} \left(I,\;\dot{H}_{r_{1} }^{s} \right)}^{\left\lceil s\right\rceil }+\left\| v\right\| _{L^{\gamma \left(r_{1} \right)} \left(I,\;\dot{H}_{r_{1} }^{s} \right)}^{\left\lceil s\right\rceil } \right)\left\| u-v\right\| _{L^{\bar{\gamma }\left(r_{1} \right)} \left(I,\;L^{a_{1} } \right)}^{\sigma +1-\left\lceil s\right\rceil},
\end{split}\end{eqnarray}
where
\begin{equation} \nonumber
2<r_{i} <\frac{2n}{n-2},~\frac{1}{a_{i}}=\frac{1}{r_{i}}-\frac{s}{n}, ~0<\bar{\gamma }\left(r_{i} \right)<\gamma \left(r_{i} \right),~i=0,\;1.
\end{equation}
\end{lem}
\begin{proof}
We only prove (3.70) whose proof is more complicated than that of (3.69). (3.69) is proved similarly. Putting $B=B\left(0,\;1\right)=\left\{x\in \R^{n} ;\;|x|\le 1\right\}$, we have
\[\left\| |x|^{-b} f(u)-|x|^{-b} f(v)\right\| _{S'\left(I,\;\dot{H}^{s} \right)} \le C_{1} +C_{2} ,\]
where
\begin{equation}\label{GrindEQ__3_71_}
C_{1} ={\mathop{\inf }\limits_{\left(\gamma \left(r\right),\;r\right)\in A}} \left\| \chi_{B^{C}}|x|^{-b} (f(u)-f(v))\right\| _{L^{\gamma \left(r\right)^{{'} } } \left(I,\;\dot{H}_{r'}^{s}\right)},
\end{equation}
\begin{equation}\label{GrindEQ__3_72_}
C_{2} ={\mathop{\inf }\limits_{\left(\gamma \left(r\right),\;r\right)\in A}} \left\|\chi_{B} |x|^{-b} (f(u)-f(v))\right\| _{L^{\gamma \left(r\right)^{{'} } } \left(I,\;\dot{H}_{r'}^{s}\right)} .
\end{equation}
and $A=\left\{\left(\gamma \left(r\right),\;r\right):\;\left(\gamma \left(r\right),\;r\right)\textrm{ is admissible}\right\}$.

\noindent First, we estimate $C_{1} $. Putting $r_{0} =\frac{n(\sigma +2)}{n+\sigma s}$, we can see that $2<r_{0} <\frac{2n}{n-2n} $. Note that
\[\frac{1}{r_{0}'}=\frac{\sigma}{a_{0}}+\frac{1}{r_{0}}=\frac{\sigma}{a_{0}}+\frac{s}{n}.\]
Using Lemma 2.1, Lemma 3.2, Lemma 3.4 and Remark 3.5, we have
\begin{eqnarray}\begin{split} \label{GrindEQ__3_73_}
\left\|\chi_{B^{C}} |x|^{-b} (f(u)-f(v))\right\| _{\dot{H}_{r'_{0} }^{s}}\lesssim&\left\|\chi_{B^{C}} |x|^{-b}\right\|_{\infty}\left\|f(u)-f(v)\right\|_{\dot{H}_{r'_{0} }^{s}}\\
&+\left\|\chi_{B^{C}} |x|^{-b}\right\|_{\dot{H}_{n/s}^{s}}\left\|f(u)-f(v)\right\|_{a_{0}/(\sigma+1)}\\
\lesssim&\left\| u-v\right\| _{a_{0} }^{\sigma -\left\lceil s\right\rceil +1} \left(\left\| u\right\| _{\dot{H}_{r_{0} }^{s} }^{\left\lceil s\right\rceil } +\left\| v\right\| _{\dot{H}_{r_{0} }^{s} }^{\left\lceil s\right\rceil } \right)\\
&+\left(\left\| u\right\| _{\dot{H}_{r_{0} }^{s} }^{\sigma } +\left\| v\right\| _{\dot{H}_{r_{0} }^{s} }^{\sigma } \right)\left\| u-v\right\| _{\dot{H}_{r_{0} }^{s} }.
\end{split}\end{eqnarray}
On the other hand, we can see that
\begin{equation} \label{GrindEQ__3_74_}
\frac{1}{\gamma \left(r_{0} \right)^{{'} } } =\frac{\sigma +1}{\gamma \left(r_{0} \right)} +\frac{4-\sigma \left(n-2s\right)}{4} .
\end{equation}
Putting $\theta _{1} :=\frac{4-\sigma \left(n-2s\right)}{4} $ and using the hypothesis $\alpha <\frac{4-2b}{n-2s} $, we can see that $\theta _{1} >0$. Putting
\[\frac{1}{\bar{\gamma }\left(r_{0} \right)} :=\frac{1}{\gamma \left(r_{0} \right)} +\frac{\theta _{1} }{\sigma +1-\left\lceil s\right\rceil } ,\]
we can also see that $0<\bar{\gamma }\left(r_{0} \right)<\gamma \left(r_{0} \right)$ and
\begin{equation} \label{GrindEQ__3_75_}
\frac{1}{\gamma \left(r_{0} \right)^{{'} } } =\frac{\left\lceil s\right\rceil }{\gamma \left(r_{0} \right)} +\frac{\sigma +1-\left\lceil s\right\rceil }{\bar{\gamma }\left(r_{0} \right)} .
\end{equation}
Using (3.73)--(3.75) and H\"{o}lder inequality, we have
\begin{eqnarray}\begin{split}\label{GrindEQ__3_76_}
C_{1} &\le \left\| \chi_{B^{C}}|x|^{-b}(f(u)-f(v))\right\| _{L^{\gamma \left(r_{0} \right)^{{'} } } \left(I,\;\dot{H}_{r'_{0} }^{s}\right)}\\
&\lesssim T^{\theta _{1} } \left(\left\| u\right\| _{L^{\gamma \left(r_{0} \right)} \left(I,\;\dot{H}_{r_{0} }^{s} \right)}^{\sigma } +\left\| v\right\| _{L^{\gamma \left(r_{0} \right)} \left(I,\;\dot{H}_{r_{0} }^{s} \right)}^{\sigma } \right)\left\| u-v\right\| _{L^{\gamma \left(r_{0} \right)} \left(I,\;\dot{H}_{r_{0} }^{s} \right)}  \\
&~~+\left(\left\| u\right\| _{L^{\gamma \left(r_{0} \right)} \left(I,\;\dot{H}_{r_{0} }^{s} \right)}^{\left\lceil s\right\rceil } +\left\| v\right\| _{L^{\gamma \left(r_{0} \right)} \left(I,\;\dot{H}_{r_{0} }^{s} \right)}^{\left\lceil s\right\rceil } \right)\left\| u-v\right\| _{L^{\bar{\gamma }\left(r_{0} \right)} \left(I,\;L^{a_{0} } \right)}^{\sigma +1-\left\lceil s\right\rceil }.
\end{split}\end{eqnarray}
Next, we estimate $C_{2} $. Putting $\bar{r}=\frac{2n}{n-2} $, it is obvious that $\left(\gamma \left(\bar{r}\right),\;\bar{r}\right)$ is admissible. Then we can take $2<r_{1} <\frac{2n}{n-2} $ satisfying the following system (see the proof of Lemma 3.5 of \cite{AK21}):
\begin{equation} \label{GrindEQ__3_77_}
\left\{\begin{array}{l} {\frac{1}{\bar{r}'} -\sigma \left(\frac{1}{r_{1} } -\frac{s}{n} \right)-\frac{1}{r_{1} } >\frac{b}{n} ,~} \\ {\frac{1}{\gamma \left(\bar{r}\right)^{{'} } } \;-\frac{\sigma +1}{\gamma \left(r_{1} \right)} >0,} \\ {\frac{1}{r_{1} } >\frac{s}{n} .} \end{array}\right.
\end{equation}
Using the same argument as in the proof of Lemma 3.5 of \cite{AK21}, it follows from Lemma 3.2, Lemma 3.4 and Remark 3.5 that
\begin{eqnarray}\begin{split} \label{GrindEQ__3_78_}
\left\| \chi_{B} |x|^{-b} (f(u)-f(v))\right\| _{\dot{H}_{\bar{r}'}^{s}}\lesssim&\left\| u-v\right\| _{a_{1} }^{\sigma -\left\lceil s\right\rceil +1}\left(\left\| u\right\| _{\dot{H}_{r_{1} }^{s} }^{\left\lceil s\right\rceil }+\left\| v\right\| _{\dot{H}_{r_{1} }^{s} }^{\left\lceil s\right\rceil } \right)\\
&+\left(\left\| u\right\| _{\dot{H}_{r_{1} }^{s} }^{\sigma } +\left\| v\right\| _{\dot{H}_{r_{1} }^{s} }^{\sigma } \right)\left\| u-v\right\| _{\dot{H}_{r_{1} }^{s}},
\end{split}\end{eqnarray}
where $\frac{1}{a_{1} } =\frac{1}{r_{1} } -\frac{s}{n} $. On the other hand, putting
\[\theta _{2} :=\frac{1}{\gamma \left(\bar{r}\right)^{{'} } } \;-\frac{\sigma +1}{\gamma \left(r_{1} \right)} , ~\frac{1}{\bar{\gamma }\left(r_{1} \right)} :=\frac{1}{\gamma \left(r_{1} \right)} +\frac{\theta _{2} }{\sigma +1-\left\lceil s\right\rceil } ,\]
and using (3.77), we can see that $\theta _{2} >0$, $0<\bar{\gamma }\left(r_{1} \right)<\gamma \left(r_{1} \right)$ and
\begin{equation} \label{GrindEQ__3_79_}
\frac{1}{\gamma \left(\bar{r}\right)^{{'} } } \;=\frac{\sigma +1}{\gamma \left(r_{1} \right)} +\theta _{2} =\frac{\left\lceil s\right\rceil }{\gamma \left(r_{1} \right)} +\frac{\sigma +1-\left\lceil s\right\rceil }{\bar{\gamma }\left(r_{1} \right)} .
\end{equation}
Using (3.78), (3.79) and H\"{o}lder inequality, we have
\begin{eqnarray}\begin{split} \label{GrindEQ__3_80_}
C_{2}&\le \left\| \chi_{B} |x|^{-b} (f(u)-f(v))\right\| _{L^{\gamma \left(\bar{r} \right)^{{'} } } \left(I,\;\dot{H}_{\bar{r}'}^{s} \right)}\\
&\lesssim T^{\theta _{2} } \left(\left\| u\right\| _{L^{\gamma \left(r_{1} \right)} \left(I,\;\dot{H}_{r_{1} }^{s} \right)}^{\sigma } +\left\| v\right\| _{L^{\gamma \left(r_{1} \right)} \left(I,\;\dot{H}_{r_{1} }^{s} \right)}^{\sigma } \right)\left\| u-v\right\| _{L^{\gamma \left(r_{1} \right)} \left(I,\;\dot{H}_{r_{1} }^{s} \right)}\\
&~~+\left(\left\| u\right\| _{L^{\gamma \left(r_{1} \right)} \left(I,\;\dot{H}_{r_{1} }^{s} \right)}^{\left\lceil s\right\rceil } +\left\| v\right\| _{L^{\gamma \left(r_{1} \right)} \left(I,\;\dot{H}_{r_{1} }^{s} \right)}^{\left\lceil s\right\rceil } \right)\left\| u-v\right\| _{L^{\bar{\gamma }\left(r_{1} \right)} \left(I,\;L^{a_{1} } \right)}^{\sigma +1-\left\lceil s\right\rceil }.
\end{split}\end{eqnarray}
Thus we have
\begin{eqnarray}\begin{split}\nonumber
&\left\| |x|^{-b} f(u)-|x|^{-b} f(v)\right\| _{S'\left(I,\;\dot{H}^{s} \right)}\le C_{1} +C_{2}\\
&~~~~~~~~~~~~~~~\lesssim\left(T^{\theta _{1} } +T^{\theta _{2} } \right)\left(\left\| u\right\| _{S\left(I,\;\dot{H}^{s} \right)}^{\sigma } +\left\| v\right\| _{S\left(I,\;\dot{H}^{s} \right)}^{\sigma } \right)\left\| u-v\right\| _{S\left(I,\;\dot{H}^{s} \right)} +Z\left(u,\, v\right),
\end{split}\end{eqnarray}
where
\begin{eqnarray}\begin{split} \label{GrindEQ__3_81_}
Z\left(u,\, v\right)&=\left(\left\| u\right\| _{L^{\gamma \left(r_{1} \right)} \left(I,\;\dot{H}_{r_{1} }^{s} \right)}^{\left\lceil s\right\rceil } +\left\| v\right\| _{L^{\gamma \left(r_{1} \right)} \left(I,\;\dot{H}_{r_{1} }^{s} \right)}^{\left\lceil s\right\rceil } \right)\left\| u-v\right\| _{L^{\bar{\gamma }\left(r_{1} \right)} \left(I,\;L^{a_{1} } \right)}^{\sigma +1-\left\lceil s\right\rceil }\\
&~~+\left(\left\| u\right\| _{L^{\gamma \left(r_{1} \right)} \left(I,\;\dot{H}_{r_{1} }^{s} \right)}^{\left\lceil s\right\rceil } +\left\| v\right\| _{L^{\gamma \left(r_{1} \right)} \left(I,\;\dot{H}_{r_{1} }^{s} \right)}^{\left\lceil s\right\rceil } \right)\left\| u-v\right\| _{L^{\bar{\gamma }\left(r_{1} \right)} \left(I,\;L^{a_{1} } \right)}^{\sigma +1-\left\lceil s\right\rceil } .
\end{split}\end{eqnarray}
\end{proof}
\begin{lem}\label{lem 3.7.}
Let $n=1,\, 2$, $0<s<\frac{n}{2} $, $0<b<n-s$ and $0<\sigma <\frac{4-2b}{n-2s} $. Assume that $f$ is of class $C\left(\sigma ,s,b\right)$.
Then we have
\begin{eqnarray}\begin{split} \label{GrindEQ__3_82_}
&\left\| |x|^{-b} f(u)-|x|^{-b} f(v)\right\| _{S'\left(I,\;\dot{H}^{s} \right)}\\
&~~~~~~~~~~\lesssim\left(T^{\theta _{1} } +T^{\theta _{2} } \right)\left(\left\| u\right\| _{S\left(I,\;\dot{H}^{s} \right)}^{\sigma } +\left\| v\right\| _{S\left(I,\;\dot{H}^{s} \right)}^{\sigma } \right)\left\| u-v\right\| _{S\left(I,\;\dot{H}^{s} \right)} ,
\end{split}\end{eqnarray}
where $I=\left[-T,\, T\right]$ and $\theta _{1} ,\, \theta _{2} >0$.
\end{lem}
\begin{proof}
Since $s<1$, it follows from the definition of class $C\left(\sigma ,s,b\right)$ that $\sigma \ge \left\lceil s\right\rceil $. Hence, combining the argument used in the proof of Lemma 3.6 of \cite{AK21} with that of Lemma 3.6, we can prove Lemma 3.7 and we omitted the details.\end{proof}

Now we are ready to prove (3.2) in the case $0<s<n/2$.

\begin{proof}[\textbf{Proof of (3.2) in the case:}] $0<s<n/2$.

Put $I=\left[-T,\, T\right]$. Since $\left\| u_{m} \right\| _{S\left(I,\;H^{s} \right)} $ is bounded, there exists
\begin{equation} \label{GrindEQ__3_83_}
M=\left\| u\right\| _{S\left(I,\;H^{s} \right)} +{\mathop{\sup }\limits_{m\ge 1}} \left\| u_{m} \right\| _{S\left(I,\;H^{s} \right)} <\infty
\end{equation}
such that $\left\| u\right\| _{S\left(I,\;H^{s} \right)} \le M$ and $\left\| u_{m} \right\| _{S\left(I,\;H^{s} \right)} \le M$. It follows from (3.1), Lemma 2.8 (Strichartz estimates) that
\begin{equation} \label{GrindEQ__3_84_}
\left\| u_{m} -u\right\| _{S\left(I,\;H^{s} \right)} \lesssim\left\| \phi _{m} -\phi \right\| _{H^{s} } +\left\| |x|^{-b} f\left(u_{m} \right)-|x|^{-b} f(u)\right\| _{S'\left(I,\;H^{s} \right)} .
\end{equation}

\textbf{Case 1.} if $f\left(z\right)$ is a polynomial in $z$ and $\bar{z}$, or if not we assume further that $\sigma \ge \left\lceil s\right\rceil $, then it follows from (3.110) in \cite{AK21}, (3.84), Lemma 3.6 and Lemma 3.7 that
\begin{equation} \label{GrindEQ__3_85_}
\left\| u_{m} -u\right\| _{S\left(I,\;H^{s} \right)} \lesssim\left\| \phi _{m} -\phi \right\| _{H^{s} } +\left(T^{\theta _{1} } +T^{\theta _{2} } \right)M^{\sigma } \left\| u_{m} -u\right\| _{S\left(I,\;H^{s} \right)} .
\end{equation}
If we take $T>0$ such that $\left(T^{\theta _{1} } +T^{\theta _{2} } \right)M^{\sigma } \le 1/2$, we can deduce from (3.85) that as $m\to \infty $,
\[\left\| u_{m} -u\right\| _{S\left(I,\;H^{s} \right)} \lesssim\left\| \phi _{m} -\phi \right\| _{H^{s} } \to 0,\]
so the solution flow is locally Lipschitz.

\textbf{Case 2.} if $f$ is not a polynomial, $\left\lceil s\right\rceil >\sigma > \left\lceil s\right\rceil -1$ and $n\ge 3$, then it follows from (3.110) in \cite{AK21} and Lemma 3.6 that
\begin{eqnarray}\begin{split} \label{GrindEQ__3_86_}
&\left\| |x|^{-b} f\left(u_{m} \right)-|x|^{-b} f(u)\right\| _{S'\left(I,\;H^{s} \right)}\\
&~~~~~\lesssim\left(T^{\theta _{1} } +T^{\theta _{2} } \right)\left(\left\| u_{m} \right\| _{S\left(I,\;H^{s} \right)}^{\sigma } +\left\| u\right\| _{S\left(I,\;H^{s} \right)}^{\sigma } \right)\left\| u_{m} -u\right\| _{S\left(I,\;H^{s} \right)} \\
&~~~~~~+\left(\left\| u_{m} \right\| _{L^{\gamma \left(r_{0} \right)} \left(I,\;H_{r_{0} }^{s} \right)}^{\left\lceil s\right\rceil } +\left\| u\right\| _{L^{\gamma \left(r_{0} \right)} \left(I,\;H_{r_{0} }^{s} \right)}^{\left\lceil s\right\rceil } \right)\left\| u_{m} -u\right\| _{L^{\bar{\gamma }\left(r_{0} \right)} \left(I,\;L^{a_{0} } \right)}^{\sigma +1-\left\lceil s\right\rceil }\\
&~~~~~~+\left(\left\| u_{m} \right\| _{L^{\gamma \left(r_{1} \right)} \left(I,\;H_{r_{1} }^{s} \right)}^{\left\lceil s\right\rceil } +\left\| u\right\| _{L^{\gamma \left(r_{1} \right)} \left(I,\;H_{r_{1} }^{s} \right)}^{\left\lceil s\right\rceil } \right)\left\| u_{m} -u\right\| _{L^{\bar{\gamma }\left(r_{1} \right)} \left(I,\;L^{a_{1} } \right)}^{\sigma +1-\left\lceil s\right\rceil } .
\end{split}\end{eqnarray}
For $i=0,\, 1$, we can take $\eta _{i} >0$ sufficiently small such that
\begin{equation} \label{GrindEQ__3_87_}
2<\eta _{i} +r_{i} <\frac{2n}{n-2} ,~\bar{\gamma }\left(r_{i} \right)<\gamma \left(\eta _{i} +r_{i} \right)<\gamma \left(r_{i} \right),~s-\frac{\eta _{i} n}{r_{i} \left(r_{i} +\eta _{i} \right)} >0,
\end{equation}
since $2<r_{i} <\frac{2n}{n-2} $, $\bar{\gamma }\left(r_{i} \right)<\gamma \left(r_{i} \right)$ and $s>0$. It follows from Theorem 1.7, (3.87) and the embedding $H_{r_{i} +\eta }^{s-\frac{\eta _{i} n}{r_{i} \left(r_{i} +\eta _{i} \right)} } \subset L^{a_{i} } $ that as $m\to \infty $,
\begin{equation} \label{GrindEQ__3_88_}
\left\| u_{m} -u\right\| _{L^{\bar{\gamma }\left(r_{i} \right)} \left(I,\;L^{a_{i} } \right)} \lesssim T^{\alpha _{i} } \left\| u_{m} -u\right\| _{L^{\gamma \left(r_{i} +\eta _{i} \right)} \left(I,\, H_{r_{i} +\eta _{i} }^{s-\frac{\eta _{i} n}{r_{i} \left(r_{i} +\eta _{i} \right)} } \right)} \to 0,
\end{equation}
where $\alpha _{i} =\frac{1}{\bar{\gamma }\left(r_{i} \right)} -\frac{1}{\gamma \left(\eta _{i} +r_{i} \right)} >0$. If we take $T>0$ such that $C\left(T^{\theta _{1} } +T^{\theta _{2} } \right)M^{\sigma } \le 1/2$, we can deduce from (3.86) and (3.88) that $\left\| u_{m} -u\right\| _{S\left(I,\;H^{s} \right)} \to 0$, as $m\to \infty $.
\end{proof}

\noindent
\subsection{Case 2. $\frac{n}{2} \le s<\min \left\{\frac{n}{2} +1,\, n\right\}$}

\begin{lem}\label{lem 3.8.}
Let $n=1$, $2<r<\infty $, $\frac{n}{2} \le s<n$ and $\sigma \ge 1$. Assume that $f\in C^{2} \left(\C\to \C\right)$ satisfies
\begin{equation} \label{GrindEQ__3_89_}
\left|f^{\left(k\right)} \left(u\right)\right|\lesssim\left|u\right|^{\sigma +1-k} ,
\end{equation}
for any $0\le k\le 2$ and $u\in \C$. Then we have
\begin{equation} \label{GrindEQ__3_90_}
\left\| f(u)-f(v)\right\| _{\dot{H}_{r}^{s} } \lesssim\left(\left\| u\right\| _{\dot{H}_{r}^{s} }^{\sigma } +\left\| v\right\| _{\dot{H}_{r}^{s} }^{\sigma } \right)\left\| u-v\right\| _{\dot{H}_{r}^{s} } .
\end{equation}
\end{lem}
\begin{proof} Since $r>2$, it follows from Lemma 2.5 that $H_{r}^{s}\subset L^{\infty}$. We use the argument similar to that used in Lemma 3.1 and we only sketch the proof. Since
\[f(u)-f(v)=\left(u-v\right)\int _{0}^{1}f'\left(v+t\left(u-v\right)\right)dt ,\]
it follows from (3.89), Lemma 2.1 (fractional product rule) and the embedding $H_{r}^{s}\subset L^{\infty}$ that
\begin{eqnarray}\begin{split}\label{GrindEQ__3_91_}
&\left\| f(u)-f(v)\right\| _{\dot{H}_{r}^{s} }=\left\| \left(u-v\right)\int _{0}^{1}f'\left(v+t\left(u-v\right)\right)dt \right\| _{\dot{H}_{r}^{s} }  \\
&~~~~~~~\lesssim \left\| \int _{0}^{1}f'\left(v+t\left(u-v\right)\right)dt \right\| _{\infty } \left\| u-v\right\| _{\dot{H}_{r}^{s} }+\left\| \int _{0}^{1}f'\left(v+t\left(u-v\right)\right)dt \right\| _{\dot{H}_{r}^{s} } \left\| u-v\right\| _{\infty } \\
&~~~~~~~\lesssim \left\| u-v\right\| _{H_{r}^{s} } \int _{0}^{1}\left\| f'\left(v+t\left(u-v\right)\right)\right\| _{H_{r}^{s}}dt
\end{split}\end{eqnarray}
Meanwhile, it follows from Lemma 2.3 (fractional chain rule), (3.89) and the embedding $H_{r}^{s}\subset L^{\infty}$ that
\begin{eqnarray}\begin{split}\nonumber
\left\| f'\left(v+t\left(u-v\right)\right)\right\| _{\dot{H}_{r}^{s} }\lesssim\left\| f''\left(v+t\left(u-v\right)\right)\right\| _{\infty } \left\| v+t\left(u-v\right)\right\| _{\dot{H}_{r}^{s}}\lesssim \left\| u\right\| _{H_{r}^{s} }^{\sigma } +\left\| v\right\| _{H_{r}^{s} }^{\sigma },
\end{split}\end{eqnarray}
for any $0\le t \le1$.
It also follows from (3.89) that
\begin{eqnarray}\begin{split}\nonumber
\left\| f'\left(v+t\left(u-v\right)\right)\right\| _{r}\lesssim\left\|\left| u\right|^{\sigma } +\left| v\right|^{\sigma }\right\|_{r}\lesssim\left\| u\right\| _{H_{r}^{s} }^{\sigma } +\left\| v\right\| _{H_{r}^{s} }^{\sigma }.
\end{split}\end{eqnarray}
for any $0\le t \le1$. Thus we have
\begin{eqnarray}\begin{split}\label{GrindEQ__3_92_}
\left\| f'\left(v+t\left(u-v\right)\right)\right\| _{H_{r}^{s} }\lesssim\left\| u\right\| _{H_{r}^{s} }^{\sigma } +\left\| v\right\| _{H_{r}^{s} }^{\sigma },
\end{split}\end{eqnarray}
for any $0\le t \le1$. In view of (3.91) and (3.92), we have (3.90).
\end{proof}
\begin{lem}\label{lem 3.9.}
Let $n\ge2$, $\max \left\{1,\, \frac{n}{2} \right\} \le s < \min \left\{n,\, \frac{n}{2} +1\right\}$ and $\sigma >\left\lceil s\right\rceil -1$. Assume that $f\in C^{\left\lceil s\right\rceil } \left(\C\to \C\right)$ satisfies $(3.15)$ and $(3.16)$.
Suppose also that one of the following conditions is satisfied:

$\cdot$ $s\in\N$ and $2<r<\infty$,

$\cdot$ $s \notin \N$, $n\ge 3$ and $2<r<n$,

$\cdot$ $s \notin \N$, $n=2$ and $r=2$.\\
Then we have
\begin{eqnarray}\begin{split} \label{GrindEQ__3_93_}
\left\| f(u)-f(v)\right\| _{\dot{H}_{r}^{s} } \lesssim&\left\| u-v\right\| _{\infty }^{\min \{ \sigma -\left\lceil s\right\rceil +1,\;1\} } \left(\left\| u\right\| _{H_{r}^{s} }^{\max \{ \left\lceil s\right\rceil ,\;\sigma \} } +\left\| v\right\| _{H_{r}^{s} }^{\max \{ \left\lceil s\right\rceil ,\;\sigma \} } \right)\\
&+\left(\left\| u\right\| _{H_{r}^{s} }^{\sigma } +\left\| v\right\| _{H_{r}^{s} }^{\sigma } \right)\left\| u-v\right\| _{H_{r}^{s} }.
\end{split}\end{eqnarray}
Moreover, if $\sigma \ge \left\lceil s\right\rceil $, then we have
\begin{equation} \label{GrindEQ__3_94_}
\left\| f(u)-f(v)\right\| _{\dot{H}_{r}^{s} } \lesssim\left(\left\| u\right\| _{H_{r}^{s} }^{\sigma } +\left\| v\right\| _{H_{r}^{s} }^{\sigma } \right)\left\| u-v\right\| _{H_{r}^{s} } .
\end{equation}
\end{lem}
\begin{proof}
Using the same argument as in the proof of Lemma 3.2, it suffices to estimate $\left\| I_{1} \right\| _{\dot{H}_{r}^{v} }$ and $ \left\| II_{1} \right\| _{\dot{H}_{r}^{v} }$,
where $v=s-\left[s\right]$ and $I_{1} $, $II_{1} $ are given by
\[I_{1} =\left(f^{\left(q\right)} \left(u\right)-f^{\left(q\right)} \left(v\right)\right)\prod _{i=1}^{q}D^{\alpha _{i} } u ,~II_{1} =f^{\left(q\right)} \left(v\right)\left(\prod _{i=1}^{q}D^{\alpha _{i} } u -\prod _{i=1}^{q}D^{\alpha _{i} } v \right),\]
for $\left|\alpha _{1} \right|+\cdots +\left|\alpha _{q} \right|=\left[s\right]$, $\left|\alpha _{i} \right|\ge 1$, $\left[s\right]\ge q\ge 1$.
For $1\le i\le q$, putting $\frac{1}{b_{i} } :=\frac{\left|\alpha _{i} \right|}{rs} $, we can see that
\[\left|\alpha _{i} \right|-\frac{n}{b_{i} } =\frac{\left|\alpha _{i} \right|}{s} \left(s-\frac{n}{r} \right)\le s-\frac{n}{r}, \]
and $r\le b_{i} $. Hence using Corollary 2.6, we have the embedding $H_{r}^{s} \subset H_{b_{i} }^{\left|\alpha _{i} \right|} $.

We divide the study in two cases: $s\in\N$ and $s\notin \N$.

\textbf{Case 1.} $s\in\N$.
Since $r>2$, there holds the embedding $H_{r}^{s}\subset L^{\infty}$.

\emph{Step 1.1.} First, we estimate $\left\| I_{1} \right\| _{r} $. Using H\"{o}lder inequality, we have
\begin{eqnarray}\begin{split}\nonumber
\left\| I_{1} \right\| _{r}&=\left\| \left(f^{\left(q\right)} \left(u\right)-f^{\left(q\right)} \left(v\right)\right)\prod _{i=1}^{q}D^{\alpha _{i} } u \right\| _{r} \le \left\| f^{\left(q\right)} \left(u\right)-f^{\left(q\right)} \left(v\right)\right\| _{\infty } \prod _{i=1}^{q}\left\| D^{\alpha _{i} } u\right\| _{b_{i}} \\
&\lesssim \left\| f^{\left(q\right)} \left(u\right)-f^{\left(q\right)} \left(v\right)\right\| _{\infty } \prod _{i=1}^{q}\left\| u\right\| _{H_{b_{i} }^{\left|\alpha _{i} \right|} }  \lesssim \left\| f^{\left(q\right)} \left(u\right)-f^{\left(q\right)} \left(v\right)\right\| _{\infty } \left\| u\right\| _{H_{r}^{s} }^{q},
\end{split}\end{eqnarray}
If $q=\left|\alpha \right|=\left\lceil s\right\rceil $, then it follows from (3.16), H\"{o}lder inequality and the embedding $H_{r}^{s} \subset L^{\infty } $ that
\begin{eqnarray}\begin{split}\nonumber
&\left\| f^{\left(\left\lceil s\right\rceil \right)} \left(u\right)-f^{\left(\left\lceil s\right\rceil \right)} \left(v\right)\right\| _{\infty } \lesssim\left\| \left|u-v\right|^{\min \{ \sigma -\left\lceil s\right\rceil +1,\;1\} } \left(\left|u\right|+\left|v\right|\right)^{\max \{ 0,\;\sigma -\left\lceil s\right\rceil \} } \right\| _{\infty } \\
&~~~~~~~~~~~~\lesssim\left\| u-v\right\| _{\infty }^{\min \{ \sigma -\left\lceil s\right\rceil +1,\;1\} } \left(\left\| u\right\| _{H_{r}^{s} }^{\max \{ 0,\;\sigma -\left\lceil s\right\rceil \} } +\left\| v\right\| _{H_{r}^{s} }^{\max \{ 0,\;\sigma -\left\lceil s\right\rceil \} } \right).
\end{split}\end{eqnarray}
If $q<\left|\alpha \right|=\left\lceil s\right\rceil $, it follows from (3.15) that
\begin{eqnarray}\begin{split}\nonumber
\left\| f^{\left(q\right)} \left(u\right)-f^{\left(q\right)} \left(v\right)\right\| _{\infty } &=\left\| \left(u-v\right)\int _{0}^{1}f^{\left(q+1\right)} \left(v+t\left(u-v\right)\right)dt \right\| _{\infty }\\
&\lesssim\left\| u-v\right\| _{\infty } \left(\left\| u\right\| _{H_{r}^{s} }^{\sigma -q} +\left\| v\right\| _{H_{r}^{s} }^{\sigma -q} \right).
\end{split}\end{eqnarray}
Thus we have
\[\left\| I_{1} \right\| _{r} \lesssim\left\| u-v\right\| _{\infty }^{\min \{ \sigma -\left\lceil s\right\rceil +1,\;1\} } \left(\left\| u\right\| _{H_{r}^{s} }^{\max \{ \left\lceil s\right\rceil ,\;\sigma \} } +\left\| v\right\| _{H_{r}^{s} }^{\max \{ \left\lceil s\right\rceil ,\;\sigma \} } \right)+\left(\left\| u\right\| _{H_{r}^{s} }^{\sigma } +\left\| v\right\| _{H_{r}^{s} }^{\sigma } \right)\left\| u-v\right\| _{H_{r}^{s} } ,\]
for any $1\le q\le s=\left\lceil s\right\rceil $.

\emph{Step 1.2.} Next, we estimate $\left\| II_{1} \right\| _{r} $. It follows from (3.30) that
\[\left\| II_{1} \right\| _{r} \le \sum _{i=1}^{q}\left\| f^{\left(q\right)} \left(v\right)\left(\prod _{j=1}^{i-1}D^{\alpha _{j} } u \prod _{j=i+1}^{q}D^{\alpha _{j} } v \left(D^{\alpha _{i} } u-D^{\alpha _{i} } v\right)\right)\right\| _{r}  .\]
It follows from (3.15), H\"{o}lder inequality and the embedding $H_{r}^{s} \subset H_{b_{i} }^{\left|\alpha _{i} \right|} $ that
\begin{eqnarray}\begin{split}\nonumber
&\left\| f^{\left(q\right)}\left(v\right)\left(\prod _{j=1}^{i-1}D^{\alpha _{j} } u \prod _{j=i+1}^{q}D^{\alpha _{j}} v \left(D^{\alpha _{i} } u-D^{\alpha _{i} } v\right)\right)\right\| _{r}\\
&~~~~~~~~~~~~~~~\lesssim \left\| u\right\|_{\infty }^{\sigma +1-q} \left\| u-v\right\| _{\dot{H}_{b_{i} }^{\left|\alpha _{i} \right|}} \prod _{j=1}^{i-1}\left\| u\right\| _{\dot{H}_{b_{j} }^{\left|\alpha _{j} \right|}}  \prod _{j=i+1}^{q}\left\| v\right\| _{\dot{H}_{b_{j} }^{\left|\alpha _{j} \right|}}\\
&~~~~~~~~~~~~~~~\lesssim\left(\left\| u\right\| _{H_{r}^{s}}^{\sigma } +\left\| v\right\| _{H_{r}^{s} }^{\sigma }\right)\left\| u-v\right\| _{H_{r}^{s}} .
\end{split}\end{eqnarray}
Thus we have
\[\left\| II_{1} \right\| _{r} \lesssim\left(\left\| u\right\| _{H_{r}^{s} }^{\sigma } +\left\| v\right\| _{H_{r}^{s} }^{\sigma } \right)\left\| u-v\right\| _{H_{r}^{s} } ,\]
this completes the proof of (3.93) in the case: $s\in\N$ and $2<r<\infty$.

\textbf{Case 2.} $s \notin \N$.
By assumption of this lemma, there are two possible cases:

$\cdot$ $s \notin \N$, $n\ge 3$ and $2<r<n$,

$\cdot$ $s \notin \N$, $n=2$ and $r=2$.\\
Since $s>n/r$ in both cases, there holds the embedding $H_{r}^{s}\subset L^{\infty}$.

\emph{Step 2.1.} First, we estimate $\left\| I_{1} \right\| _{\dot{H}_{r}^{v} } $.
It follows from Lemma 2.1 (fractional product rule) that
\begin{eqnarray}\begin{split} \label{GrindEQ__3_95_}
\left\| I_{1} \right\| _{\dot{H}_{r}^{v} }&\lesssim \left\| f^{\left(q\right)} \left(u\right)-f^{\left(q\right)} \left(v\right)\right\| _{\dot{H}_{\bar{r}_{1} }^{v} } \left\| \prod _{i=1}^{q}D^{\alpha _{i} } u \right\| _{\bar{q}_{1} } +\left\| f^{\left(q\right)} \left(u\right)-f^{\left(q\right)} \left(v\right)\right\| _{\infty } \left\| \prod _{i=1}^{q}D^{\alpha _{i} } u \right\| _{\dot{H}_{r}^{v} }\\
&\equiv I_{3} +II_{3} ,
\end{split}\end{eqnarray}
where
\begin{equation} \label{GrindEQ__3_96_}
\frac{1}{\bar{r}_{1} } =\frac{v}{rs} ,~\frac{1}{\bar{q}_{1} } =\frac{\left[s\right]}{rs} =\sum _{i=1}^{q}\frac{1}{b_{i} }  .
\end{equation}
First, we estimate $I_{3} $. Putting $\frac{1}{\bar{r}_{2} } =\frac{1}{\bar{r}_{1} } -\frac{v}{n} +\frac{1}{n} $, we can see that
\[0<\frac{1}{\bar{r}_{2} } =\frac{v}{rs} -\frac{v}{n} +\frac{1}{n} <\frac{1}{n} \le\frac{1}{r},\]
which implies that $r<\bar{r}_{2} <\infty $. Noticing that
\[1-\frac{1}{\bar{r}_{2} } =v-\frac{n}{\bar{r}_{1} } =\frac{v}{s} \left(s-\frac{n}{r} \right)<s-\frac{n}{r},~r<\bar{r}_{2},~r<\bar{r}_{1}, \]
and using Corollary 2.6, we have the embeddings: $H_{\bar{r}_{2} }^{1} \subset H_{r}^{s} $ and $H_{\bar{r}_{1} }^{v} \subset H_{r}^{s} $. Thus we have
\begin{eqnarray}\begin{split}\label{GrindEQ__3_97_}
&\left\| f^{\left(q\right)} \left(u\right)-f^{\left(q\right)} \left(v\right)\right\| _{\dot{H}_{\bar{r}_{1} }^{v} } \lesssim\left\| f^{\left(q\right)} \left(u\right)-f^{\left(q\right)} \left(v\right)\right\| _{\dot{H}_{\bar{r}_{2} }^{1} } =\sum _{i=1}^{n}\left\| \partial _{x_{i} } \left(f^{\left(q\right)} \left(u\right)-f^{\left(q\right)} \left(v\right)\right)\right\| _{\bar{r}_{2}}\\
&~~~~~=\sum _{i=1}^{n}\left\| f^{\left(q+1\right)} \left(u\right)\partial _{x_{i} } u-f^{\left(q+1\right)} \left(v\right)\partial _{x_{i} } v\right\| _{\bar{r}_{2}}\\
&~~~~~\le \sum _{i=1}^{n}\left\| \left(f^{\left(q+1\right)} \left(u\right)-f^{\left(q+1\right)} \left(v\right)\right)\partial _{x_{i} } u\right\| _{\bar{r}_{2} } +\left\| f^{\left(q+1\right)} \left(v\right)\left(\partial _{x_{i} } u-\partial _{x_{i} } v\right)\right\| _{\bar{r}_{2}}\\
&~~~~~\lesssim \sum _{i=1}^{n}\left\| f^{\left(q+1\right)} \left(u\right)-f^{\left(q+1\right)} \left(v\right)\right\| _{\infty } \left\| \partial _{x_{i} } u\right\| _{\bar{r}_{2} } +\left\| f^{\left(q+1\right)} \left(v\right)\right\| _{\infty } \left\| \partial _{x_{i} } u-\partial _{x_{i} } v\right\| _{\bar{r}_{2} }\\
&~~~~~\lesssim\sum _{i=1}^{n}\left\| f^{\left(q+1\right)} \left(u\right)-f^{\left(q+1\right)} \left(v\right)\right\| _{\infty } \left\| u\right\| _{H_{r}^{s} } +\left\| f^{\left(q+1\right)} \left(v\right)\right\| _{\infty } \left\| u-v\right\| _{H_{r}^{s} }.
\end{split}\end{eqnarray}
If $q=\left[s\right]=\left\lceil s\right\rceil -1$, it follows from (3.16), H\"{o}lder inequality and the embedding $H_{r}^{s} \subset L^{\infty } $ that
\begin{equation} \label{GrindEQ__3_98_}
\left\| f^{\left(\left\lceil s\right\rceil \right)} \left(u\right)-f^{\left(\left\lceil s\right\rceil \right)} \left(v\right)\right\| _{\infty } \lesssim\left\| u-v\right\| _{\infty }^{\min \{ \sigma -\left\lceil s\right\rceil +1,\;1\} } \left(\left\| u\right\| _{H_{r}^{s} }^{\max \{ 0,\;\sigma -\left\lceil s\right\rceil \} } +\left\| v\right\| _{H_{r}^{s} }^{\max \{ 0,\;\sigma -\left\lceil s\right\rceil \} } \right).
\end{equation}
If $q<\left[s\right]=\left\lceil s\right\rceil -1$, it follows from (3.15) that
\begin{eqnarray}\begin{split} \label{GrindEQ__3_99_}
\left\| f^{\left(q+1\right)} \left(u\right)-f^{\left(q+1\right)} \left(v\right)\right\| _{\infty } &=\left\| \left(u-v\right)\int _{0}^{1}f^{\left(q+2\right)} \left(v+t\left(u-v\right)\right)dt \right\| _{\infty }\\
&\lesssim\left\| u-v\right\| _{H_{r}^{s} } \left(\left\| u\right\| _{H_{r}^{s} }^{\sigma -q-1} +\left\| v\right\| _{H_{r}^{s} }^{\sigma -q-1} \right).
\end{split}\end{eqnarray}
We can also deduce from (3.15) that
\begin{equation} \label{GrindEQ__3_100_}
\left\| f^{\left(q+1\right)} \left(v\right)\right\| _{\infty } \le \left\| v\right\| _{\infty }^{\sigma -q} \lesssim\left\| v\right\| _{H_{r}^{s} }^{\sigma -q} .
\end{equation}
We also have
\begin{equation} \label{GrindEQ__3_101_}
\left\| \prod _{i=1}^{q}D^{\alpha _{i} } u \right\| _{\bar{q}_{1} } \le \prod _{i=1}^{q}\left\| D^{\alpha _{i} } u\right\| _{b_{i} }  \lesssim\left\| u\right\| _{H_{r}^{s} }^{q} .
\end{equation}
In view of (3.97)--(3.101), we have
\begin{eqnarray}\begin{split}\nonumber
I_{3}&=\left\| f^{\left(q\right)} \left(u\right)-f^{\left(q\right)} \left(v\right)\right\| _{\dot{H}_{\bar{r}_{1} }^{v} } \left\| \prod _{i=1}^{q}D^{\alpha _{i} } u \right\| _{\bar{q}_{1} } \\
&\lesssim\left\| u-v\right\| _{\infty }^{\min \{ \sigma -\left\lceil s\right\rceil +1,\;1\} } \left(\left\| u\right\| _{H_{r}^{s} }^{\max \{ \left\lceil s\right\rceil ,\;\sigma \} } +\left\| v\right\| _{H_{r}^{s} }^{\max \{ \left\lceil s\right\rceil ,\;\sigma \} } \right)+\left(\left\| u\right\| _{\dot{H}_{r}^{s} }^{\sigma } +\left\| v\right\| _{H_{r}^{s} }^{\sigma } \right)\left\| u-v\right\| _{H_{r}^{s} }.
\end{split}\end{eqnarray}
Combining the argument used in the case 1 with that used in the proof of Lemma 3.9 of \cite{AK21}, we can prove
\begin{eqnarray}\begin{split}\nonumber
II_{3}&=\left\| f^{\left(q\right)} \left(u\right)-f^{\left(q\right)} \left(v\right)\right\| _{\infty } \left\| \prod _{i=1}^{q}D^{\alpha _{i} } u \right\| _{\dot{H}_{r}^{v}}\lesssim \left(\left\| u\right\| _{H_{r}^{s} }^{\sigma } +\left\| v\right\| _{H_{r}^{s} }^{\sigma }\right)\left\| u-v\right\| _{H_{r}^{s} },
\end{split}\end{eqnarray}
whose proof will be omitted.
Thus we have
\begin{eqnarray}\begin{split}\nonumber
\left\| I_{1} \right\| _{\dot{H}_{r}^{v} }&\lesssim I_{3} +II_{3}\\
&\lesssim\left\| u-v\right\| _{\infty }^{\min \{ \sigma -\left\lceil s\right\rceil +1,\;1\} } \left(\left\| u\right\| _{H_{r}^{s} }^{\max \{ \left\lceil s\right\rceil ,\;\sigma \} } +\left\| v\right\| _{H_{r}^{s} }^{\max \{ \left\lceil s\right\rceil ,\;\sigma \} } \right)+\left(\left\| u\right\| _{H_{r}^{s} }^{\sigma } +\left\| v\right\| _{H_{r}^{s} }^{\sigma } \right)\left\| u-v\right\| _{H_{r}^{s} }
\end{split}\end{eqnarray}

\emph{Step 2.2.} Next, we estimate $\left\| II_{1} \right\| _{\dot{H}_{p}^{v} } $.
 It follows from Lemma 2.1 (fractional product rule) that
\begin{equation} \label{GrindEQ__3_102_}
\left\| II_{1} \right\| _{\dot{H}_{r}^{v} } \lesssim \left\| f^{\left(q\right)} \left(v\right)\right\| _{\dot{H}_{\bar{r}_{1} }^{v} } \left\| \prod _{i=1}^{q}D^{\alpha _{i} } u -\prod _{i=1}^{q}D^{\alpha _{i} } v \right\| _{\bar{q}_{1} } +\left\| f^{\left(q\right)} \left(v\right)\right\| _{\infty } \left\| \prod _{i=1}^{q}D^{\alpha _{i} } u -\prod _{i=1}^{q}D^{\alpha _{i} } v \right\| _{\dot{H}_{r}^{v} } ,
\end{equation}
where $\bar{r}_{1} $ and $\bar{q}_{1} $ are given in (3.96). Using Lemma 2.3 (fractional chain rule), (3.15) and the embedding $H_{\bar{r}_{1} }^{v} \subset H_{r}^{s} $, we have
\begin{equation} \label{GrindEQ__3_103_}
\left\| f^{\left(q\right)} \left(v\right)\right\| _{\dot{H}_{\bar{r}_{1} }^{v} } \lesssim\left\| f^{\left(q+1\right)} \left(v\right)\right\| _{\infty } \left\| v\right\| _{\dot{H}_{\bar{r}_{1} }^{v} } \lesssim\left\| v\right\| _{H_{r}^{s} }^{\sigma -q+1} .
\end{equation}
It also follows from (3.30) that
\begin{equation} \label{GrindEQ__3_104_}
\left\| \prod _{i=1}^{q}D^{\alpha _{i} } u -\prod _{i=1}^{q}D^{\alpha _{i} } v \right\| _{\bar{q}_{1} } \le \sum _{i=1}^{q}\left\| \prod _{j=1}^{i-1}D^{\alpha _{j} } u \prod _{j=i+1}^{q}D^{\alpha _{j} } v \left(D^{\alpha _{i} } u-D^{\alpha _{i} } v\right)\right\| _{\bar{q}_{1} }  .
\end{equation}
For any $1\le i\le q$, it follows from H\"{o}lder inequality and the embedding $H_{r}^{s} \subset H_{b_{i} }^{\left|\alpha _{i} \right|} $ that
\begin{equation} \label{GrindEQ__3_105_}
\left\| \prod _{j=1}^{i-1}D^{\alpha _{j} } u \prod _{j=i+1}^{q}D^{\alpha _{j} } v \left(D^{\alpha _{i} } u-D^{\alpha _{i} } v\right)\right\| _{\bar{q}_{1} } \lesssim\left(\left\| u\right\| _{H_{r}^{s} }^{q-1} +\left\| v\right\| _{H_{r}^{s} }^{q-1} \right)\left\| u-v\right\| _{H_{r}^{s} } .
\end{equation}
In view of (3.103)--(3.105), we have
\begin{equation} \label{GrindEQ__3_106_}
\left\| f^{\left(q\right)} \left(v\right)\right\| _{\dot{H}_{\bar{r}_{1} }^{v} } \left\| \prod _{i=1}^{q}D^{\alpha _{i} } u -\prod _{i=1}^{q}D^{\alpha _{i} } v \right\| _{\bar{q}_{1} } \lesssim\left(\left\| u\right\| _{H_{r}^{s} }^{\sigma } +\left\| v\right\| _{H_{r}^{s} }^{\sigma } \right)\left\| u-v\right\| _{H_{r}^{s} } .
\end{equation}
Combining the above argument with that used to estimate $\bar{B}_{2} $ in Lemma 3.9 of \cite{AK21}, we can also prove
\begin{equation} \label{GrindEQ__3_107_}
\left\| f^{\left(q\right)} \left(v\right)\right\| _{\infty } \left\| \prod _{i=1}^{q}D^{\alpha _{i} } u -\prod _{i=1}^{q}D^{\alpha _{i} } v \right\| _{\dot{H}_{r}^{v} } \lesssim\left(\left\| u\right\| _{H_{r}^{s} }^{\sigma } +\left\| v\right\| _{H_{r}^{s} }^{\sigma } \right)\left\| u-v\right\| _{H_{r}^{s} } ,
\end{equation}
whose proof will be omitted. In view of (3.102), (3.106) and (3.107), we have
\[\left\| II_{1} \right\| _{\dot{H}_{r}^{v} } \lesssim\left\| u-v\right\| _{\infty }^{\min \{ \sigma -\left\lceil s\right\rceil +1,\;1\} } \left(\left\| u\right\| _{H_{r}^{s} }^{\max \{ \left\lceil s\right\rceil ,\;\sigma \} } +\left\| v\right\| _{H_{r}^{s} }^{\max \{ \left\lceil s\right\rceil ,\;\sigma \} } \right)+\left(\left\| u\right\| _{H_{r}^{s} }^{\sigma } +\left\| v\right\| _{H_{r}^{s} }^{\sigma } \right)\left\| u-v\right\| _{H_{r}^{s} } ,\]
this completes the proof of (3.93) in the case: $s \notin \N$.

\noindent If $\sigma \ge \left\lceil s\right\rceil $, then (3.94) follows directly from (3.93) and the embedding $H_{r}^{s} \subset L^{\infty } $.
\end{proof}

\begin{rem}\label{rem 3.10.}
\textnormal{If $f\left(z\right)$ is a polynomial in $z$ and $\bar{z}$ satisfying $1<\deg \left(f\right)=1+\sigma $, we can see that the assumption $\sigma >\left\lceil s\right\rceil -1$ in Lemma 3.9 can be removed and that there holds
\[\left\| f(u)-f(v)\right\| _{\dot{H}_{r}^{s} } \lesssim\left(\left\| u\right\| _{H_{r}^{s} }^{\sigma } +\left\| v\right\| _{H_{r}^{s} }^{\sigma } \right)\left\| u-v\right\| _{H_{r}^{s} },\]
by using the same argument as in the proof of Lemma 3.3}
\end{rem}
\begin{lem}\label{lem 3.11.}
Let $\sigma >0$. Assume that $f\in C^{1} \left(\C\to \C\right)$ satisfies \textnormal{(3.68)}. Suppose also that one of the following two conditions is satisfied:

$\cdot s\ge n/2$ and $r>2$,

$\cdot s>n/2$ and $r=2$.\\
Then we have
\[\left\| f(u)-f(v)\right\| _{r} \lesssim\left(\left\| u\right\| _{H_{r}^{s} }^{\sigma } +\left\| v\right\| _{H_{r}^{s} }^{\sigma } \right)\left\| u-v\right\| _{H_{r}^{s} } ,\]
\[\left\| f(u)-f(v)\right\| _{\infty} \lesssim\left(\left\| u\right\| _{H_{r}^{s} }^{\sigma } +\left\| v\right\| _{H_{r}^{s} }^{\sigma } \right)\left\| u-v\right\| _{H_{r}^{s} }.\]
\end{lem}
\begin{proof}
Using the same argument as in Lemma 3.4, we get
\[\left|f(u)-f(v)\right|\lesssim\left(\left|u\right|^{\sigma } +\left|v\right|^{\sigma } \right)\left|u-v\right|.\]
By using the hypothesis of this lemma and Lemma 2.5 (2), there holds the embedding $H_{r}^{s}\subset L^{\infty}$. Therefore, we have
\begin{eqnarray}\begin{split}\nonumber
\left\| f(u)-f(v)\right\| _{r}&\lesssim\left\| \left(\left|u\right|^{\sigma } +\left|v\right|^{\sigma } \right)\left(u-v\right)\right\| _{r} \le\left(\left\| u\right\| _{\infty}^{\sigma } +\left\| v\right\| _{\infty}^{\sigma } \right)\left\| u-v\right\| _{r}\\
&\lesssim\left(\left\| u\right\| _{H_{r}^{s} }^{\sigma } +\left\| v\right\| _{H_{r}^{s} }^{\sigma } \right)\left\| u-v\right\| _{H_{r}^{s} }.
\end{split}\end{eqnarray}
Similarly, we also have
\begin{eqnarray}\begin{split}\nonumber
\left\| f(u)-f(v)\right\| _{\infty}&\lesssim\left\| \left(\left|u\right|^{\sigma } +\left|v\right|^{\sigma } \right)\left(u-v\right)\right\| _{\infty} \le\left(\left\| u\right\| _{\infty}^{\sigma } +\left\| v\right\| _{\infty}^{\sigma } \right)\left\| u-v\right\| _{\infty}\\
&\lesssim\left(\left\| u\right\| _{H_{r}^{s} }^{\sigma } +\left\| v\right\| _{H_{r}^{s} }^{\sigma } \right)\left\| u-v\right\| _{H_{r}^{s} }.
\end{split}\end{eqnarray}
\end{proof}

Using Lemma 3.8, Lemma 3.9, Remark 3.10 and Lemma 3.11, we have the following important estimates of the term $|x|^{-b} f(u)-|x|^{-b} f(v)$.

\begin{lem}\label{lem 3.11.}
Let $n\ge 3$, $\frac{n}{2} \le s<\frac{n}{2} +1$, $0<b<{1}+\frac{n-2s}{2} $ and $0<\sigma <\infty $. Assume that $f$ is of class $C\left(\sigma ,s,b\right)$.

\textnormal{(1)} If $f\left(z\right)$ is a polynomial in $z$ and $\bar{z}$, or if not we assume further that $\sigma \ge \left\lceil s\right\rceil $, then we have
\begin{equation} \label{GrindEQ__3_108_}
\left\| |x|^{-b} f(u)-|x|^{-b} f(v)\right\| _{S'\left(I,\;\dot{H}^{s} \right)} \lesssim\left(T^{\theta _{1} } +T^{\theta _{2} } \right)\left(\left\| u\right\| _{S\left(I,\;H^{s} \right)}^{\sigma } +\left\| v\right\| _{S\left(I,\;H^{s} \right)}^{\sigma } \right)\left\| u-v\right\| _{S\left(I,\;H^{s} \right)} ,
\end{equation}
where $I=\left[-T,\, T\right]$ and $\theta _{1} ,\, \theta _{2} >0$.

\textnormal{(2)} If $f$ is not a polynomial and $\left\lceil s\right\rceil >\sigma > \left\lceil s\right\rceil -1$, then we have
\begin{eqnarray}\begin{split} \label{GrindEQ__3_109_}
&\left\| |x|^{-b} f(u)-|x|^{-b} f(v)\right\| _{S'\left(I,\;\dot{H}^{s} \right)}\\
&~~~~~~~~~~~~~~~~~~~~\lesssim\left(T^{\theta _{1} } +T^{\theta _{2} } \right)\left(\left\| u\right\| _{S\left(I,\;H^{s} \right)}^{\sigma } +\left\| v\right\| _{S\left(I,\;H^{s} \right)}^{\sigma } \right)\left\| u-v\right\| _{S\left(I,\;H^{s} \right)}\\
&~~~~~~~~~~~~~~~~~~~~~~~+\left(\left\| u\right\| _{L^{\gamma \left(\bar{r}_{0} \right)} \left(I,\;H_{\bar{r}_{0} }^{s} \right)}^{\left\lceil s\right\rceil } +\left\| v\right\| _{L^{\gamma \left(\bar{r}_{0} \right)} \left(I,\;H_{\bar{r}_{0} }^{s} \right)}^{\left\lceil s\right\rceil } \right)\left\| u-v\right\| _{L^{\bar{\gamma }\left(\bar{r}_{0} \right)} \left(I,\;L^{\infty } \right)}^{\sigma +1-\left\lceil s\right\rceil }\\
&~~~~~~~~~~~~~~~~~~~~~~~+\left(\left\| u\right\| _{L^{\gamma \left(\bar{r}_{1} \right)} \left(I,\;H_{\bar{r}_{1} }^{s} \right)}^{\left\lceil s\right\rceil } +\left\| v\right\| _{L^{\gamma \left(\bar{r}_{1} \right)} \left(I,\;H_{\bar{r}_{1} }^{s} \right)}^{\left\lceil s\right\rceil } \right)\left\| u-v\right\| _{L^{\bar{\gamma }\left(\bar{r}_{1} \right)} \left(I,\;L^{\infty } \right)}^{\sigma +1-\left\lceil s\right\rceil },
\end{split}\end{eqnarray}
where
\begin{equation} \nonumber
2<\bar{r}_{i} <\frac{2n}{n-2},~0<\bar{\gamma }\left(\bar{r}_{i}\right)<\gamma \left(\bar{r}_{i} \right),~i=0,~1.
\end{equation}
\end{lem}
\begin{proof}
We use the argument similar to that used in the proof of Lemma 3.11 of \cite{AK21} and we only sketch the proof. Since the proofs of (3.108) and (3.109) are similar to each other, we only prove (3.109) whose proof is more complicated than (3.108). As in the proof of Lemma 3.6, we divide the estimate in $B$ and $B^{C} $, indeed,
\[\left\| |x|^{-b} f(u)-|x|^{-b} f(v)\right\| _{S'\left(I,\;\dot{H}^{s} \right)} \le C_{1} +C_{2} ,\]
where $C_{1}$ and $C_{2}$ are given in (3.71) and (3.72), respectively.

\noindent First, we estimate $C_{1} $.
We can easily see that there exists an admissible pair $\left(\gamma \left(\bar{r}_{0} \right),\bar{r}_{0} \right)$ satisfying the following system:
\begin{equation} \label{GrindEQ__3 110_}
\left\{\begin{array}{l} {\frac{1}{2} -\frac{1}{\bar{r}_{0} } <\frac{b}{n},} \\ {\frac{1}{\gamma \left(2\right)^{{'} } } \;-\frac{\sigma +1}{\gamma \left(\bar{r}_{0} \right)} >0,} \\
{2<\bar{r}_{0} <\min \left\{\frac{2n}{n-2},\;n \right\}.} \end{array}\right.
\end{equation}
Repeating the same argument as in the estimate $\bar{C}_{1} $ in Lemma 3.11 of \cite{AK21}, it follows from (3.110), Remark 3.5, Lemma 3.9 and Lemma 3.11 that
\begin{eqnarray}\begin{split} \label{GrindEQ__3 111_}
&\left\| \chi_{B^{C}}|x|^{-b}(f(u)-f(v))\right\| _{\dot{H}^{s}}\lesssim \left\| f(u)-f(v)\right\| _{\dot{H}_{\bar{r}_{0} }^{s} } +\left\| f(u)-f(v)\right\| _{\bar{r}_{0} }\\
&~~~~~~~~~~~~~~~~~~~\lesssim \left\| u-v\right\| _{\infty }^{\sigma -\left\lceil s\right\rceil +1} \left(\left\| u\right\| _{H_{r}^{s} }^{\left\lceil s\right\rceil} +\left\| v\right\| _{H_{r}^{s} }^{\left\lceil s\right\rceil} \right)+\left(\left\| u\right\| _{H_{r}^{s} }^{\sigma } +\left\| v\right\| _{H_{r}^{s} }^{\sigma } \right)\left\| u-v\right\| _{H_{r}^{s} }.
\end{split}\end{eqnarray}
Using (3.110), (3.111) and the argument similar to that used to estimate $C_{1} $ in Lemma 3.6, we can easily get the estimate of $C_{1}$ and we omit the details.

\noindent Next, we estimate $C_{2}$. Putting $\bar{r}=\frac{2n}{n-2}$, we can easily verify that there exists an admissible pair $\left(\gamma \left(\bar{r}_{1} \right),\;\bar{r}_{1} \right)$ satisfying the following system:
\begin{equation} \label{GrindEQ__3 112_}
\left\{\begin{array}{l} {\frac{1}{\bar{r}'} -\frac{1}{\bar{r}_{1} } >\frac{b}{n},} \\ {\frac{1}{\gamma \left(\bar{r}\right)^{{'} } } \;-\frac{\sigma +1}{\gamma \left(\bar{r}_{1} \right)} >0,} \\ {2<\bar{r}_{1} <\min \left\{\frac{2n}{n-2},\;n \right\}.} \end{array}\right.
\end{equation}
By combining the argument used to estimate $\bar{C}_{2}$ in Lemma 3.11 of \cite{AK21} with the above argument, we can get the estimate of $C_{2}$ and we omit the details.
\end{proof}

Using Lemma 3.8, Lemma 3.11 and repeating the same argument as in the proof of Lemma 3.12 in \cite{AK21}, we can get the following lemma, whose proof will be omitted.
\begin{lem}\label{lem 3.13.}
Let $n=1$, $\frac{n}{2} \le s<n$, $0<b<n-s$ and $0<\sigma <\infty $. Assume that $f$ is of class $C\left(\sigma ,s,b\right)$. Then we have
\begin{equation} \nonumber
\left\| |x|^{-b} f(u)-|x|^{-b} f(v)\right\| _{S'\left(I,\;\dot{H}^{s} \right)} \lesssim\left(T^{\theta _{1} } +T^{\theta _{2} } \right)\left(\left\| u\right\| _{S\left(I,\;H^{s} \right)}^{\sigma } +\left\| v\right\| _{S\left(I,\;H^{s} \right)}^{\sigma } \right)\left\| u-v\right\| _{S\left(I,\;H^{s} \right)} ,
\end{equation}
where $I=\left[-T,\, T\right]$ and $\theta _{1} ,\, \theta _{2} >0$.
\end{lem}

Similarly, using Lemma 3.9 and Lemma 3.11, we also have the following Lemma, whose proof will be omitted.
\begin{lem}\label{lem 3.14.}
Let $n=2$, $s=1$, $0<b<n-s$ and $0<\sigma <\infty $. Assume that $f$ is of class $C\left(\sigma ,s,b\right)$.

\textnormal{(1)} If $\sigma \ge 1$, then we have
\begin{equation} \nonumber
\left\| |x|^{-b} f(u)-|x|^{-b} f(v)\right\| _{S'\left(I,\;\dot{H}^{s} \right)} \lesssim\left(T^{\theta _{1} } +T^{\theta _{2} } \right)\left(\left\| u\right\| _{S\left(I,\;H^{s} \right)}^{\sigma } +\left\| v\right\| _{S\left(I,\;H^{s} \right)}^{\sigma } \right)\left\| u-v\right\| _{S\left(I,\;H^{s} \right)} ,
\end{equation}
where $I=\left[-T,\, T\right]$ and $\theta _{1} ,\, \theta _{2} >0$.

\textnormal{(2)} If $0<\sigma<1$, then we have
\begin{eqnarray}\begin{split} \nonumber
&\left\| |x|^{-b} f(u)-|x|^{-b} f(v)\right\| _{S'\left(I,\;\dot{H}^{s} \right)}\\
&~~~~~~~~~~~\lesssim\left(T^{\theta _{1} } +T^{\theta _{2} } \right)\left(\left\| u\right\| _{S\left(I,\;H^{s} \right)}^{\sigma } +\left\| v\right\| _{S\left(I,\;H^{s} \right)}^{\sigma } \right)\left\| u-v\right\| _{S\left(I,\;H^{s} \right)}\\
&~~~~~~~~~~~~~+\left(\left\| u\right\| _{L^{\gamma \left(\bar{r}_{2} \right)} \left(I,\;H_{\bar{r}_{2} }^{s} \right)}^{\left\lceil s\right\rceil } +\left\| v\right\| _{L^{\gamma \left(\bar{r}_{2} \right)} \left(I,\;H_{\bar{r}_{2} }^{s} \right)}^{\left\lceil s\right\rceil } \right)\left\| u-v\right\| _{L^{\bar{\gamma }\left(\bar{r}_{2} \right)} \left(I,\;L^{\infty } \right)}^{\sigma +1-\left\lceil s\right\rceil }\\
&~~~~~~~~~~~~~+\left(\left\| u\right\| _{L^{\gamma \left(\bar{r}_{3} \right)} \left(I,\;H_{\bar{r}_{3} }^{s} \right)}^{\left\lceil s\right\rceil } +\left\| v\right\| _{L^{\gamma \left(\bar{r}_{3} \right)} \left(I,\;H_{\bar{r}_{3} }^{s} \right)}^{\left\lceil s\right\rceil } \right)\left\| u-v\right\| _{L^{\bar{\gamma }\left(\bar{r}_{3} \right)} \left(I,\;L^{\infty } \right)}^{\sigma +1-\left\lceil s\right\rceil },
\end{split}\end{eqnarray}
where
\begin{equation} \nonumber
2<\bar{r}_{i} <\infty ,~0<\bar{\gamma }\left(\bar{r}_{i} \right)<\gamma \left(\bar{r}_{i} \right),~ i=2,~3.
\end{equation}
\end{lem}
\begin{lem}\label{lem 3.15.}
Let $n=2$, $1<s<2$, $0<b<n-s$ and $0<\sigma <\infty $. Assume that $f$ is of class $C\left(\sigma ,s,b\right)$.

\textnormal{(1)} If $f\left(z\right)$ is a polynomial in $z$ and $\bar{z}$, or if not we assume further that $\sigma \ge 2$, then we have
\begin{equation} \nonumber
\left\| |x|^{-b} f(u)-|x|^{-b} f(v)\right\| _{S'\left(I,\;\dot{H}^{s} \right)} \lesssim\left(T^{\theta _{1} } +T^{\theta _{2} } \right)\left(\left\| u\right\| _{S\left(I,\;H^{s} \right)}^{\sigma } +\left\| v\right\| _{S\left(I,\;H^{s} \right)}^{\sigma } \right)\left\| u-v\right\| _{S\left(I,\;H^{s} \right)} ,
\end{equation}
where $I=\left[-T,\, T\right]$ and $\theta _{1} ,\, \theta _{2} >0$.

\textnormal{(2)} If $f$ is not a polynomial and $1<\sigma<2$, then we have
\begin{eqnarray}\begin{split} \nonumber
&\left\| |x|^{-b} f(u)-|x|^{-b} f(v)\right\| _{S'\left(I,\;\dot{H}^{s} \right)}\lesssim\left(T^{\theta _{1} } +T^{\theta _{2} } \right)\left(\left\| u\right\| _{S\left(I,\;H^{s} \right)}^{\sigma } +\left\| v\right\| _{S\left(I,\;H^{s} \right)}^{\sigma } \right)\left\| u-v\right\| _{S\left(I,\;H^{s} \right)}\\
&~~~~~~~~~~~~~~~~~~~~~~~+\left(\left\| u\right\| _{L^{\infty} \left(I,\;H^{s} \right)}^{\left\lceil s\right\rceil } +\left\| v\right\| _{L^{\infty} \left(I,\;H^{s} \right)}^{\left\lceil s\right\rceil } \right)\left\| u-v\right\| _{L^{\beta_{1}} \left(I,\;L^{\infty } \right)\bigcap L^{\beta_{2}} \left(I,\;L^{\infty } \right)}^{\sigma +1-\left\lceil s\right\rceil }.
\end{split}\end{eqnarray}
where $0<\beta_{1},\; \beta_{2}<\infty$.
\end{lem}
\begin{proof}
We only consider the case (2). The case (1) is treated similarly.
As in the above lemmas, we divide the estimate in $B$ and $B^{C} $, indeed,
\[\left\| |x|^{-b} f(u)-|x|^{-b} f(v)\right\| _{S'\left(I,\;\dot{H}^{s} \right)} \le C_{1} +C_{2} ,\]
where $C_{1}$ and $C_{2}$ are given in (3.71) and (3.72) respectively.
First, we estimate $C_{1}$. It is obvious that $\chi_{B^{C}}|x|^{-b}\in L^{\infty}$. We can also see that $\chi_{B^{C}}|x|^{-b}\in \dot{H}^{s}$, since $\frac{n}{2}<s<b+s$.
It follows from Lemma 2.1, Lemma 3.9 and Lemma 3.11 that
\begin{eqnarray}\begin{split} \label{GrindEQ__3 113_}
&\left\|\chi_{B^{C}}|x|^{-b}(f(u)-f(v))\right\| _{\dot{H}^{s}}\\
&~~~~~\lesssim \left\|\chi_{B^{C}}|x|^{-b}\right\|_{L^{\infty} }\left\| f(u)-f(v)\right\| _{\dot{H}^{s}} +\left\| \chi_{B^{C}}|x|^{-b}\right\|_{\dot{H}^{s}}\left\| f(u)-f(v)\right\| _{\infty}\\
&~~~~~\lesssim \left\| u-v\right\| _{\infty }^{\sigma-1} \left(\left\| u\right\| _{H^{s} }^{2} +\left\| v\right\| _{H^{s} }^{2} \right)+\left(\left\| u\right\| _{H^{s} }^{\sigma } +\left\| v\right\| _{H^{s} }^{\sigma } \right)\left\| u-v\right\| _{H^{s} }.
\end{split}\end{eqnarray}
Using (3.113) and H\"{o}lder inequality, we have
\begin{eqnarray}\begin{split} \label{GrindEQ__3 114_}
C_{1}&\le \left\|\chi_{B^{C}}|x|^{-b}(f(u)-f(v))\right\| _{L^{\gamma \left(2\right)^{{'} } } \left(I,\;\dot{H}^{s}\right)}\\
&\lesssim T\left(\left\| u\right\| _{L^{\infty}\left(I,\;H^{s} \right)}^{\sigma } +\left\| v\right\| _{L^{\infty}\left(I,\;H^{s} \right)}^{\sigma } \right)\left\| u-v\right\| _{L^{\infty}\left(I,\;H^{s} \right)}\\
&~~~+\left(\left\| u\right\| _{L^{\infty} \left(I,\;H^{s} \right)}^{2} +\left\| v\right\| _{L^{\infty} \left(I,\;H^{s} \right)}^{2} \right)\left\| u-v\right\| _{L^{\beta_{1}} \left(I,\;L^{\infty } \right)}^{\sigma-1},
\end{split}\end{eqnarray}
where $\beta_{1}=\sigma-1$.
Next, we estimate $C_{2}$.
Since $b+s<n$, we can take $\hat{r}>2$ large enough such that
\begin{equation} \nonumber
\frac{b+s}{n} +\frac{1}{\hat{r}} <1.
\end{equation}
Putting $\frac{1}{p}=\frac{1}{\hat{r}'}-\frac{1}{2}$, we can see that
$$
\frac{1}{p}=\frac{1}{\hat{r}'}-\frac{1}{2}>\frac{b+s}{n}-\frac{1}{2}>\frac{b}{n},
$$
which implies that $\chi_{B}|x|^{-b} \in L^{p}$. Since $\frac{1}{\hat{r}'}-\frac{b+s}{n}>0$, we can see that $\chi_{B}|x|^{-b}\in \dot{H}_{\hat{r}'}^{s}$.
It follows from Lemma 2.1, Lemma 3.9 and Lemma 3.11 that
\begin{eqnarray}\begin{split} \label{GrindEQ__3 115_}
&\left\| \chi_{B}|x|^{-b} (f(u)-f(v))\right\| _{\dot{H}_{\hat{r}'}^{s}}\\
&~~~~~\lesssim \left\|\chi_{B}|x|^{-b}\right\|_{L^{p}}\left\| f(u)-f(v)\right\| _{\dot{H}^{s}} +\left\|\chi_{B}|x|^{-b}\right\|_{{\dot{H}_{\hat{r}'}^{s} }}\left\| f(u)-f(v)\right\| _{\infty}\\
&~~~~~\lesssim \left\| u-v\right\| _{\infty }^{\sigma-1} \left(\left\| u\right\| _{H^{s} }^{2} +\left\| v\right\| _{H^{s} }^{2} \right)+\left(\left\| u\right\| _{H^{s} }^{\sigma } +\left\| v\right\| _{H^{s} }^{\sigma } \right)\left\| u-v\right\| _{H^{s} }.
\end{split}\end{eqnarray}
Using (3.115) and H\"{o}lder inequality, we have
\begin{eqnarray}\begin{split} \label{GrindEQ__3 116_}
C_{2}&\le \left\| |x|^{-b} f(u)-|x|^{-b} f(v)\right\| _{L^{\gamma \left(\hat{r}\right)^{{'} } } \left(I,\;{\dot{H}_{\hat{r}'}^{s} \left(B\right)}\right)}\\
&\lesssim T^{\theta _{2}}\left(\left\| u\right\| _{L^{\infty}\left(I,\;H^{s} \right)}^{\sigma } +\left\| v\right\| _{L^{\infty}\left(I,\;H^{s} \right)}^{\sigma } \right)\left\| u-v\right\| _{L^{\infty}\left(I,\;H^{s} \right)}\\
&~~~+\left(\left\| u\right\| _{L^{\infty} \left(I,\;H^{s} \right)}^{2} +\left\| v\right\| _{L^{\infty} \left(I,\;H^{s} \right)}^{2} \right)\left\| u-v\right\| _{L^{\beta_{1}} \left(I,\;L^{\infty } \right)}^{\sigma-1},
\end{split}\end{eqnarray}
where $\theta _{2}=\frac{1}{\gamma \left(\hat{r}\right)^{{'} }}$ and $\beta_{2}=\left(\sigma-1\right)\gamma \left(\hat{r}\right)^{{'} }$.
\end{proof}
Now we are ready to prove (3.2) in the case $\frac{n}{2} \le s<\min \left\{\frac{n}{2} +1,\, n\right\}$.

\begin{proof}[\textbf{Proof of (3.2) in the case:}]$\frac{n}{2} \le s<\min \left\{\frac{n}{2} +1,\, n\right\}$

We use the argument similar to that used in the case $s<n/2$ and we only sketch the proof.

\textbf{Case 1.} if $f\left(z\right)$ is a polynomial in $z$ and $\bar{z}$, or if not we assume further that $\sigma \ge \left\lceil s\right\rceil $, then there are possible four cases:

$\cdot n\ge3$ and $n/2\le s<1+n/2$,

$\cdot n=2$ and $s=1$,

$\cdot n=2$ and $1<s<2$,

$\cdot n=1$ and $1/2\le s<1$.

In each case, using one of Lemma 3.12--Lemma 3.15 and the same argument as in the case $s<n/2$, we can easily see that the solution flow is locally Lipschitz and we omit the details.

\textbf{Case 2.} if $f$ is not a polynomial and $\left\lceil s\right\rceil >\sigma > \left\lceil s\right\rceil -1$, then there are possible three cases:

$\cdot n\ge3$ and $n/2\le s<1+n/2$,

$\cdot n=2$ and $s=1$,

$\cdot n=2$ and $1<s<2$.

First, we consider the case $n\ge3$. It follows from Lemma 3.11 of \cite{AK21} and Lemma 3.12 that
\begin{eqnarray}\begin{split} \label{GrindEQ__3_117_}
&\left\| |x|^{-b} f\left(u_{m} \right)-|x|^{-b} f(u)\right\| _{S'\left(I,\;H^{s} \right)}\\
&~~~~~~~~\lesssim\left(T^{\theta _{1} } +T^{\theta _{2} } \right)\left(\left\| u_{m} \right\| _{S\left(I,\;H^{s} \right)}^{\sigma } +\left\| u\right\| _{S\left(I,\;H^{s} \right)}^{\sigma } \right)\left\| u_{m} -u\right\| _{S\left(I,\;H^{s} \right)}\\
&~~~~~~~~~~~+\left(\left\| u_{m} \right\| _{L^{\gamma \left(\bar{r}_{0} \right)} \left(I,\;H_{\bar{r}_{0} }^{s} \right)}^{\left\lceil s\right\rceil } +\left\| u\right\| _{L^{\gamma \left(\bar{r}_{0} \right)} \left(I,\;H_{\bar{r}_{0} }^{s} \right)}^{\left\lceil s\right\rceil } \right)\left\| u_{m} -u\right\| _{L^{\bar{\gamma }\left(\bar{r}_{0} \right)} \left(I,\;L^{\infty } \right)}^{\sigma +1-\left\lceil s\right\rceil }\\
&~~~~~~~~~~~+\left(\left\| u_{m} \right\| _{L^{\gamma \left(\bar{r}_{1} \right)} \left(I,\;H_{\bar{r}_{1} }^{s} \right)}^{\left\lceil s\right\rceil } +\left\| u\right\| _{L^{\gamma \left(\bar{r}_{1} \right)} \left(I,\;H_{\bar{r}_{1} }^{s} \right)}^{\left\lceil s\right\rceil } \right)\left\| u_{m} -u\right\| _{L^{\bar{\gamma }\left(\bar{r}_{1} \right)} \left(I,\;L^{\infty } \right)}^{\sigma +1-\left\lceil s\right\rceil },
\end{split}\end{eqnarray}
where
$$
2<\bar{r}_{i} <\frac{2n}{n-2},~0<\bar{\gamma }\left(\bar{r}_{i}\right)<\gamma \left(\bar{r}_{i} \right),~i=0,~1.
$$
For $i=0,\, 1$, we can take $\bar{\eta }_{i} >0$ sufficiently small such that
\begin{equation} \label{GrindEQ__3_118_}
2<\bar{\eta }_{i} +\bar{r}_{i} <\frac{2n}{n-2},~\bar{\gamma }\left(\bar{r}_{i} \right)<\gamma \left(\bar{\eta }_{i} +\bar{r}_{i} \right)<\gamma \left(\bar{r}_{i} \right),~s-\frac{\bar{\eta }_{i} n}{\bar{r}_{i} \left(\bar{r}_{i} +\bar{\eta }_{i} \right)} >0.
\end{equation}
Since
\[s-\frac{\bar{\eta }_{i} n}{\bar{r}_{i} \left(\bar{r}_{i} +\bar{\eta }_{i} \right)} -\frac{n}{\bar{r}_{i} +\bar{\eta }_{i} } =s-\frac{n}{\bar{r}_{i} } >s-\frac{n}{2} \ge 0,\]
it follows from Lemma 2.5 (2) that $H_{\bar{r}_{i} +\bar{\eta }_{i} }^{s-\frac{\bar{\eta }_{i} n}{\bar{r}_{i} \left(\bar{r}_{i} +\bar{\eta }_{i} \right)} } \subset L^{\infty } $. Thus using (3.118) and Theorem 1.7, we can see that as $m\to \infty $,
\begin{equation} \label{GrindEQ__3_119_}
\left\| u_{m} -u\right\| _{L^{\bar{\gamma }\left(r_{i} \right)} \left(I,\;L^{\infty } \right)} \lesssim T^{\bar{\alpha }_{i} } \left\| u_{m} -u\right\| _{L^{\gamma \left(\bar{r}_{i} +\bar{\eta }_{i} \right)} \left(I,\, H_{\bar{r}_{i} +\bar{\eta }_{i} }^{s-\frac{\bar{\eta }_{i} n}{\bar{r}_{i} \left(\bar{r}_{i} +\bar{\eta }_{i} \right)} } \right)} \to 0,
\end{equation}
where $\bar{\alpha }_{i} =\frac{1}{\bar{\gamma }\left(\bar{r}_{i} \right)} -\frac{1}{\gamma \left(\bar{r}_{i} +\bar{\eta }_{i} \right)} >0$. Using (3.117) and (3.119), we can get the desired result and we omitted the details.

When $n=2$ and $s=1$, we can get the desired result by using Lemma 3.14 and the above argument.

Finally, we consider the case: $n=2$ and $1<s<2$. It follows from Lemma 3.12 of \cite{AK21} and Lemma 3.15 that
\begin{eqnarray}\begin{split} \label{GrindEQ__3_120_}
&\left\| |x|^{-b} f(u)-|x|^{-b} f(v)\right\| _{S'\left(I,\;\dot{H}^{s} \right)}\lesssim\left(T^{\theta _{1} } +T^{\theta _{2} } \right)\left(\left\| u\right\| _{S\left(I,\;H^{s} \right)}^{\sigma } +\left\| v\right\| _{S\left(I,\;H^{s} \right)}^{\sigma } \right)\left\| u-v\right\| _{S\left(I,\;H^{s} \right)}\\
&~~~~~~~~~~~~~~~~~~~~~~~+\left(\left\| u\right\| _{L^{\infty} \left(I,\;H^{s} \right)}^{\left\lceil s\right\rceil } +\left\| v\right\| _{L^{\infty} \left(I,\;H^{s} \right)}^{\left\lceil s\right\rceil } \right)\left\| u-v\right\| _{L^{\beta_{1}} \left(I,\;L^{\infty } \right)\bigcap L^{\beta_{2}} \left(I,\;L^{\infty } \right)}^{\sigma +1-\left\lceil s\right\rceil }.
\end{split}\end{eqnarray}
where $0<\beta_{1},\; \beta_{2}<\gamma \left(2\right)=\infty$.
We can take $\bar{\eta }_{3} >0$ sufficiently small such that
\begin{equation} \label{GrindEQ__3_121_}
2+\bar{\eta }_{3}<\infty,~\beta_{i}<\gamma \left(\bar{\eta }_{3} +2 \right)<\gamma \left(2\right),~s-\frac{\bar{\eta }_{3}n}{2\left(2 +\bar{\eta }_{3} \right)} >0,
\end{equation}
for $i=1,\;2$. Since
\[s-\frac{\bar{\eta }_{3}n}{2\left(2 +\bar{\eta }_{3}\right)}-\frac{n}{2 +\bar{\eta }_{3}} =s-\frac{n}{2} >0,\]
there holds the embedding $H_{2+\bar{\eta }_{3} }^{s-\frac{\bar{\eta }_{3} n}{2 \left(2+\bar{\eta }_{3} \right)} } \subset L^{\infty } $. Thus using (3.121) and Theorem 1.7, we can see that as $m\to \infty $,
\begin{equation} \label{GrindEQ__3_122_}
\left\| u_{m} -u\right\| _{L^{\beta_{1}} \left(I,\;L^{\infty } \right)\bigcap L^{\beta_{1}} \left(I,\;L^{\infty } \right)} \lesssim \left(T^{\bar{\beta}_{1}}+T^{\bar{\beta}_{2} }\right) \left\| u_{m} -u\right\| _{L^{\gamma \left(2 +\bar{\eta }_{3} \right)} \left(I,\, H_{2+\bar{\eta }_{3} }^{s-\frac{\bar{\eta }_{3} n}{2 \left(2+\bar{\eta }_{3} \right)} } \right)} \to 0,
\end{equation}
where $\bar{\beta }_{i} =\frac{1}{\beta_{i}} -\frac{1}{\gamma \left(2+\bar{\eta }_{3} \right)} >0$. Using (3.120) and (3.122), we can get the desired result and we omitted the details.
\end{proof}


\end{document}